\documentclass[12pt]{amsart}
\usepackage{amsfonts,amssymb, amscd, latexsym, graphicx, psfrag, color, float}

\usepackage[all,2cell]{xy}
\UseTwocells

\usepackage{mathrsfs}
\usepackage{eucal}

\usepackage[colorlinks,
             linkcolor=black,
             citecolor=blue,
             bookmarksnumbered,hyperindex,hyperfigures,pdfpagelabels]{hyperref}
\usepackage{todonotes}
\usepackage{tikz-cd}

\newtheorem{dummy}{dummy}[section]
\newtheorem{lemma}[dummy]{Lemma}
\newtheorem{theorem}[dummy]{Theorem}
\newtheorem*{theorem*}{Theorem}

\newtheorem{corollary}[dummy]{Corollary}
\newtheorem{proposition}[dummy]{Proposition}
\theoremstyle{definition}
\newtheorem{definition}[dummy]{Definition}
\newtheorem{example}[dummy]{Example}
\newtheorem{remark}[dummy]{Remark}

\newcommand{\bC}{\mathbb{C}}

\newcommand{\cA}{\mathcal{A}}
\newcommand{\cB}{\mathcal{B}}

\newcommand{\cD}{\mathcal{D}}
\newcommand{\cE}{\mathcal{E}}
\newcommand{\cF}{\mathcal{F}}
\newcommand{\cG}{\mathcal{G}}

\newcommand{\cI}{\mathcal{I}}
\newcommand{\cJ}{\mathcal{J}}

\newcommand{\cO}{\mathcal{O}}
\newcommand{\cP}{\mathcal{P}}

\newcommand{\cS}{\mathcal{S}}

\newcommand{\cU}{\mathcal{U}}
\newcommand{\cV}{\mathcal{V}}
\newcommand{\cX}{\mathcal{X}}

\newcommand{\cZ}{\mathcal{Z}}

\newcommand{\sC}{\mathscr{C}}
\newcommand{\sD}{\mathscr{D}}
\newcommand{\sE}{\mathscr{E}}

\newcommand{\Aff}{\mathbf{Aff}^{\mathit{LFT}}_{\mathbb{C}}}
\newcommand{\Affk}{\mathbf{Aff}^{\mathit{LFT}}_{k}}
\newcommand{\Affku}{\mathbf{Aff}^{\cU}_{k}}

\newcommand{\TopC}{\Top_\mathbb{C}}
\newcommand{\TopCs}{\TopC^s}
\newcommand{\Topi}{\mathfrak{Top}_\i}

\newcommand{\Spec}{\mathrm{Spec}\,}

\newcommand{\Sh}{\operatorname{Sh}}

\newcommand{\Hom}{\mathrm{Hom}}

\newcommand{\et}{\acute{e}t}

\def\Top{\mathbf{Top}}
\def\Sch{\mathbf{Sch}}

\renewcommand{\i}{\infty}
\def\iGpd{\operatorname{Gpd}_\i}

\def\Pro{\operatorname{Pro}}

\def\Fun{\operatorname{Fun}}
\DeclareMathOperator{\Psh}{Psh}
\def\Pshi{\operatorname{Psh}_\i}
\def\LPshi{\widehat{\operatorname{Psh}}_\i}
\def\Shi{\Sh_\i}
\def\Hshi{\mathbb{H}\mathrm{ypSh}_\i}
\def\colim{\underrightarrow{\mbox{colim}\vspace{0.5pt}}\mspace{4mu}}

\renewcommand{\lim}{\varprojlim\mspace{3mu}}
\def\blank{\mspace{3mu}\cdot\mspace{3mu}}

\def\Profs{\operatorname{Prof}\left(\cS\right)}
\def\Prof{\operatorname{Prof}}
\def\Shape{\mathit{Shape}}

\def\Lan{\operatorname{Lan}}

\def\Top{\mathbf{Top}}
\def\Set{\mathit{Set}}

\def\Htop{\Topi^{\mathit{Hens.}}}
\def\Algku{\mathbf{Alg}^{\cU}_{k}}
\def\Dmsch{\mathfrak{DM}\Sch^\cU_k}
\def\Dmschet{\mathfrak{DM}\Sch^{\cU,\et}_k}

\def\Aut{\mathbf{Aut}}

\def\Disc{\mbox{Disc}}

\def\longlongrightarrow{-\!\!\!-\!\!\!-\!\!\!-\!\!\!-\!\!\!-\!\!\!\longrightarrow}
\def\longlonglongrightarrow{-\!\!\!-\!\!\!-\!\!\!-\!\!\!-\!\!\!-\!\!\!\longlongrightarrow}

\newcommand*{\longhookrightarrow}{\ensuremath{\lhook\joinrel\relbar\joinrel\rightarrow}}
\newcommand*{\longlonghookrightarrow}{\ensuremath{\lhook\joinrel\relbar\joinrel\relbar\joinrel\rightarrow}}

\begin{document}

\title[On the \'etale homotopy type of higher stacks]
{On the \'etale homotopy type of higher stacks}

\author{David Carchedi}

\maketitle

\begin{abstract}
A new approach to \'etale homotopy theory is presented which applies to a much broader class of objects than previously existing approaches, namely it applies not only to all schemes (without any local Noetherian hypothesis), but also to arbitrary higher stacks on the \'etale site of such schemes, and in particular to all algebraic stacks. This approach also produces a more refined invariant, namely a pro-object in the infinity category of spaces, rather than in the homotopy category. We prove a profinite comparison theorem at this level of generality, which states that if $\cX$ is an arbitrary higher stack on the \'etale site of affine schemes of finite type over $\bC,$ then the \'etale homotopy type of $\cX$ agrees with the homotopy type of the underlying stack $\cX_{top}$ on the topological site, after profinite completion. In particular, if $\cX$ is an Artin stack locally of finite type over $\bC$, our definition of the \'etale homotopy type of $\cX$ agrees up to profinite completion with the homotopy type of the underlying topological stack $\cX_{top}$ of $\cX$ in the sense of Noohi \cite{No1}. In order to prove our comparison theorem, we provide a modern reformulation of the theory of local systems and their cohomology using the language of $\i$-categories which we believe to be of independent interest.
\end{abstract}

\section{Introduction}

Given a complex variety $V,$ one can associate topological invariants to this variety by computing invariants of its underlying topological space $V_{an},$ equipped with the complex analytic topology. However, for a variety over an arbitrary base ring, there is no good notion of underlying topological space which plays the same role. (It is well known that the Zariski topology is too coarse.) \'Etale cohomology gives a way of partly circumventing this problem, since it associates cohomology groups to a scheme, and it is well known that if $V$ is a complex variety, then its \'etale cohomology with coefficients in any finite abelian group $A$ agrees with the singular cohomology of its underlying space $V_{an}$ with the same coefficients. The \'etale homotopy type of a scheme takes things one giant step further. Although it does not associate a genuine topological space to a scheme, it associates a (pro-)homotopy type, and this allows one to associate to a scheme much more refined topological invariants, e.g. higher homotopy groups.

The original notion of \'etale homotopy type goes back to seminal work of Artin and Mazur \cite{ArtinMazur} in 1969. They give a way of associating to any locally Noetherian scheme a pro-object in the homotopy category of simplicial sets. From the \'etale homotopy type of a scheme, one can recover its \'etale cohomology, and also its \'etale fundamental group, and higher homotopy groups. \'Etale homotopy types have made many important impacts in mathematics, perhaps most famously in the proof of the Adams conjecture \cite{Quillen,sullivan,Friedlander2}. More recently, \'etale homotopy theory has been an important tool in studying the rational points of algebraic varieties \cite{harpaztomer,ambrus,torsors}, and also has an interesting connection with motivic homotopy theory \cite{realization,schmidt}.

Artin and Mazur also introduced the notion of \emph{profinite completion}, which was motivated by the notion of profinite completion of groups, and they proved the following celebrated comparison theorem:

\begin{theorem}\cite[Theorem 12.9]{ArtinMazur}
Let $X$ be a pointed connected scheme of finite type over $\bC$. Then there is a canonical map $$\left[X_{an}\right] \to \left[X_{\et}\right]$$ from the homotopy type of $X_{an}$ to the \'etale homotopy type of $X$ which induces an isomorphism on profinite completions.
\end{theorem}

The above theorem is a vast generalization of the comparison theorem for \'etale cohomology.

Although \'etale homotopy theory, as developed by Artin and Mazur, has been quite a successful endeavor, there are limitations to their framework. The most serious limitation is that, although their notion of \'etale homotopy type naturally extends to Deligne-Mumford stacks, it does not easily extend to more general objects, such as Artin stacks. A more subtle limitation is related to the notion of homotopy coherence: the \'etale homotopy type of a scheme in the sense of Artin and Mazur produces a pro-object--- a diagram of a certain shape --- in the homotopy category of spaces (or simplicial sets);  it is well known that a diagram in the homotopy category need not lift to a diagram of actual spaces--- this is the issue of homotopy coherence. A third limitation is that 
the schemes in question are required to be locally Noetherian. This excludes many natural examples. One such example, is that Vakil and Wickelgren show in \cite{VW} that for a quasicompact and quasiseparated scheme $X,$ there exists a universal cover $\widetilde{X},$ which itself surprisingly is always a scheme, however it may fail to be locally Noetherian even when $X$ is, so one cannot apply the machinery of Artin and Mazur to it. The first two limitations were partly remedied by subsequent work of Friedlander \cite{Friedlander}, as he refined the construction of Artin and Mazur to define the \'etale homotopy type of a locally Noetherian scheme as a pro-object in the actual category of simplicial sets, rather than its homotopy category. He also extended the construction to simplicial schemes, and proved a comparison theorem similar to the above, but for pointed connected simplicial schemes of finite type over $\bC.$ Unfortunately, the profinite comparison result that Friedlander proves uses the same notion of profinite completion as Artin-Mazur, which happens at the level of pro-objects in the homotopy category of spaces, and also Friedlander's approach still has a locally Noetherian hypothesis.

In this paper, we present a new approach to \'etale homotopy theory which offers a refinement of the original construction which produces a pro-object
in the $\i$-category of spaces rather than its homotopy category and applies
to a much broader class of objects, namely to arbitrary higher stacks on the \'etale site of affine schemes over an arbitrary base, with no Noetherian hypothesis. Furthermore, our approach to profinite completion follows Lurie and produces a pro-object in the $\i$-category of $\pi$-finite spaces (Definition \ref{dfn:pi-finite}). We also prove a generalization of Artin and Mazur's comparison theorem which holds for any higher stack on the \'etale site of affine schemes of finite type over $\bC.$ In particular, the comparison result holds for any algebraic stack locally of finite type over $\bC,$ or more generally, any $n$-geometric stack locally of finite type over $\bC$, in the sense of \cite{TV2}.

As our machinery applies to much more general objects than previous frameworks, it is ripe for future applications. One application which has already been explored is related to the profinite homotopy type of log schemes, and is explained in Section 7 of \cite{knhom}. The results of \cite{knhom} imply that if $X$ is a (fine saturated) log scheme locally of finite type over $\bC,$ then the homotopy type of its Kato-Nakayama space \cite{KN} agrees after profinite completion with the homotopy type of the underlying topological stack of its infinite root stack--- which is a pro-algebraic stack \cite{TV}. Our comparison theorem, Theorem \ref{thm: main}, implies that both of these profinite homotopy types also agree with the profinite \'etale homotopy type of the infinite root stack. As the later makes sense for log schemes over a more general base, this gives a suitable replacement for the Kato-Nakayama space in positive characteristics. The previously existing comparison theorems were not robust enough to apply in this situation.

The machinery and formulation of our approach is quite different than the work of Artin-Mazur and Friedlander, however for locally Noetherian schemes, our definition of \'etale homotopy type turns out to be essentially the same after unwinding the definitions. The principal difference between the definition of the \'etale homotopy type of such a scheme as computed according to our machinery and its definition computed according to the machinery of Artin-Mazur is two-fold, namely our approach uses \v{C}ech covers and theirs uses hypercovers, and our definition yields a pro-object in the $\i$-category of spaces, and theirs yields a pro-object in the \emph{homotopy category} of spaces (see Remark \ref{rmk: closely related}). By a recent result of Hoyois \cite{Hoyois}, for locally Noetherian schemes, the former is the only real difference between the definition using our approach and the definition using the approach of Friedlander (see Proposition \ref{prop: Hoyois}). The difference between using \v{C}ech covers as opposed to using  hypercovers is moreover erased by passing to profinite homotopy types.

Our approach necessitates the use of the powerful framework of $\i$-categories. The use of this language makes our definition of \'etale homotopy type much more simple and transparent than previous definitions. This is partly due to the fact that the language of $\i$-categories allows for a simple definition of pro-spaces and profinite spaces. At the same time, this approach to pro-spaces is equivalent to the approach of Edwards-Hastings and Isaksen \cite{edwardshastings,Isak1} using model categories \cite{prohomotopy}, and moreover, the $\i$-categorical approach to profinite spaces is equivalent to Quick's model-categorical approach as in \cite{Qu,Qu2}. However, the $\i$-categorical approach to pro-spaces and profinite spaces is much easier to work with, e.g. (c.f. \cite{dagxiii}): If $\cS$ is the $\i$-category of spaces, succinctly, the $\i$-category $\Pro\left(\cS\right)$ of \emph{pro-spaces} is the full subcategory of $\Fun\left(\cS,\cS\right)^{op}$ - the opposite of the $\i$-category of functors from spaces to spaces - on those functors which preserve finite limits (and are accessible). 

Let $\cX$ be an arbitrary higher stack on the \'etale site of affine schemes over $k.$ Let $\cG$ be an arbitrary space in $\cS$ (i.e. an $\i$-groupoid). Denote by $\Delta^{\et}\left(\cG\right)$ the stackification of the constant presheaf with value $\cG.$ Then, as a pro-space, the \'etale homotopy type $\Pi^{\et}_\i\left(\cX\right)$ of $\cX$ as a functor from spaces to spaces sends the space $\cG$ to the space of maps $$\Hom\left(\cX,\Delta^{\et}\left(\cG\right)\right).$$
This assignment produces a functor
$$\Shi\left(\mathbf{Aff}_k\right) \stackrel{\Pi^{\et}_\i}{\longlongrightarrow} \Pro\left(\cS\right)$$
from higher stacks on the \'etale site to pro-spaces, sending a stack $\cX$ to its \'etale homotopy type.

Strictly speaking, the above description of the \'etale homotopy type of an arbitrary stack is not by definition; this description is the content of Theorem \ref{thm: concrete}. Our definition is of more geometric origin, as our approach has its roots in the philosophy  of Grothendieck which led to the inception of the concept of a topos; topoi were invented in order to associate to a scheme an underlying ``space'' whose cohomology is (by definition) the \emph{\'etale cohomology} of the scheme in question, and this was an important first step towards producing a Weil cohomology theory and proving the Weil conjectures. We take seriously the idea that the correct geometric ``space'' underlying a scheme $X$ is its small \'etale topos $\Sh\left(X_{\et}\right),$ and therefore all the topological invariants of a scheme $X$ should actually be invariants of the topos $\Sh\left(X_{\et}\right).$ To make this precise, one needs a way of associating to a topos a pro-homotopy type. This can be accomplished by using the theory of $\i$-topoi. Indeed, given an $\i$-topos, there is a simple construction, originally due to To\"{e}n and Vezzosi, which associates to an $\i$-topos $\cE$ a pro-space $\Shape\left(\cE\right)$ called its \emph{shape}, which is to be the thought of as the pro-homotopy type of the $\i$-topos in question, and as any topos can be in a natural way regarded as an $\i$-topos, this gives a way of associating to any scheme a pro-homotopy type.

The above discussion works well for schemes. It also works well for Deligne-Mumford stacks, as they can be modeled geometrically as ringed topoi. However, to extend the definition of \'etale homotopy type to an arbitrary higher stack, one needs a new idea. We accomplish this by formally extending the functor associating to a scheme its small \'etale $\i$-topos $\Shi\left(X_{\et}\right)$ (the $\i$-topos associated to its small \'etale topos) to a colimit preserving functor $$\Shi\left(\mathbf{Aff}_k,\et\right) \to \mathfrak{Top}_\i$$ from the $\i$-category of higher stacks on the \'etale site of affine schemes over $k$ to the $\i$-category of $\i$-topoi. The \'etale homotopy type of a higher stack is then defined to be the shape of its associated $\i$-topos via the above functor. Although this definition is not very tractable for stacks which are not Deligne-Mumford, this is rectified by Theorem \ref{thm: concrete}, and moreover, in light of our comparison theorem, Theorem \ref{thm: main}, it is still a reasonable definition for Artin stacks, which may a priori be non-obvious due to the use of the \'etale topology rather than say the smooth topology.

There is substantial overlap of our results with those of Chough, whose manuscript \cite{chough} is still in preparation. Although our overall approaches are quite different, we believe both approaches will prove valuable to the mathematical community.

\subsection{The comparison theorem}
We will now explain in detail the content of our comparison theorem:\\

In \cite{knhom}, we extend two important classical constructions for schemes and topological spaces to higher stacks, namely the analytification functor and the functor sending a topological space to its homotopy type.\\

\textbf{Analytification:}
Consider the classical analytification functor
$$\left(\blank\right)_{an}:\Sch^{LFT}_\bC \to \Top,$$
from schemes locally of finite type over $\bC$ to topological spaces. It sends a scheme $X$ to its space of $\bC$-points equipped with the complex analytic topology. Motivated by the desire to associate to an algebraic stack over $\bC$ a natural topological object from which one can extract topological invariants, Noohi extends this construction in \cite{No1} to a functor
$$\left(\blank\right)_{top}:\mathbf{A}\!\mathfrak{lgSt}^{LFT}_{\mathbb{C}} \to \mathfrak{TopSt}$$ from Artin stacks locally of finite type over $\mathbb{C}$ to topological stacks. In \cite{knhom}, we extend this further to a colimit preserving functor
$$\left(\blank\right)_{top}:\Shi\left(\Aff,\mbox{\'et}\right) \to \Hshi\left(\TopC\right)$$
from $\i$-sheaves on the \'etale site of affine schemes of finite type over $\bC,$ to hypersheaves on a suitable category $\TopC$ of topological spaces.\\

\textbf{The homotopy type of a stack:}
In \cite{No2}, Noohi defines a functor
$$ho:\mathfrak{TopSt} \to Ho\left(\Top\right)$$
from the $2$-category of topological stacks to the homotopy category of topological spaces, sending a topological stack $\cX$ to its weak homotopy type. Explicitly, if $\cG$ is a topological groupoid presentation for $\cX,$ $ho\left(\cX\right)$ has the weak homotopy type of the classifying space of $\cG.$ In \cite{No3} Noohi and Coyne refine this to a functor to the $\i$-category of spaces $\cS.$

In \cite{knhom}, we extend this further to a colimit preserving functor
$$\Pi_\i:\Hshi\left(\TopC\right) \to \cS$$ from the $\i$-category of hypersheaves on $\TopC$ to the $\i$-category of spaces.

The final important construction we need in order to explain our comparison result is\\

\textbf{Profinite completion:}
In \cite{dagxiii}, Lurie constructs the profinite completion functor
$$\widehat{\left(\blank\right)}:\cS \to \Prof\left(\cS\right)$$
from the $\i$-category of spaces to the $\i$-category of profinite spaces. In fact, this is the restriction of a profinite completion functor $$\Pro\left(\cS\right) \to \Prof\left(\cS\right)$$ from pro-spaces to profinite spaces, and composing this functor with our \'etale homotopy type functor $$\Shi\left(\mathbf{Aff}_k\right) \stackrel{\Pi^{\et}_\i}{\longlongrightarrow} \Pro\left(\cS\right)$$ produces a functor
$$\Shi\left(\mathbf{Aff}_k\right) \stackrel{\widehat{\Pi}^{\et}_\i}{\longlongrightarrow} \Prof\left(\cS\right)$$ which sends a stack $\cX$ to its \emph{profinite \'etale homotopy type.}\\

We now state our main result:\\

\begin{theorem}\label{thm: main intro}
Let $\Aff$ denote the category of affine schemes of finite type over $\bC.$ The following diagram commutes up to equivalence:
$$\xymatrix@C=2.5cm@R=2cm{\Shi\left(\Aff,\mbox{\'et}\right) \ar[r]^-{\widehat{\Pi}^{\et}_\i} \ar[d]_-{\left(\blank\right)_{top}} & \Profs\\
\Hshi\left(\TopC\right) \ar[r]^-{\Pi_\i} & \cS. \ar[u]_-{\widehat{\left(\blank\right)}}}$$
In particular, for any $\i$-sheaf $F$ on $\left(\Aff,\mbox{\'et}\right),$ there is an equivalence of profinite spaces $$\widehat{\Pi}^{\et}_\i\left(F\right) \simeq \widehat{\Pi}_\i\left(F_{top}\right),$$ between the profinite \'etale homotopy type of $F$ and the profinite completion of the homotopy type of the underlying stack $F_{top}$ on $\TopC$.
\end{theorem}

This theorem has the following immediate corollary:

\begin{corollary}
Let $\cX$ be an Artin stack locally of finite type over $\bC,$ then there is an equivalence of profinite spaces $$\widehat{\Pi}^{\et}_\i\left(\cX\right) \simeq \widehat{\Pi}_\i\left(\cX_{top}\right),$$ between the profinite \'etale homotopy type of $\cX$ and the profinite completion of the homotopy type of the underlying topological stack $\cX_{top}$.
\end{corollary}
\subsection{Overview of our approach}
It turns out that all of the functors in the statement of Theorem \ref{thm: main intro} are colimit preserving, so, by the Yoneda lemma, the comparison result for higher stacks in fact follows formally from the comparison result for affine schemes. Therefore, in order to prove Theorem \ref{thm: main intro}, one must prove an analogue of Artin and Mazur's classical comparison theorem for (affine) schemes of finite type over $\bC$ (without the additional assumptions of being pointed or connected) and for our new $\i$-categorical definition of \'etale homotopy type.

Our strategy is close in spirit to the original strategy of Artin and Mazur, but uses more modern machinery. The key ideas are the following:

\begin{itemize}
\item[1)] For $X$ a separated scheme of finite type over $\bC,$ the shape of the $\i$-topos $\Shi\left(X_{an}\right)$ of $\i$-sheaves on its underlying space $X_{an},$ $\Shape\left(\Shi\left(X_{an}\right)\right)$ is canonically equivalent to the underlying homotopy type $\Pi_\i\left(X_{an}\right)$ of $X_{an}.$
\item[2)] Analytification induces a geometric morphism of topoi
$$\varepsilon:\Sh\left(X_{an}\right) \to \Sh\left(X_{\et}\right)$$ from the topos of sheaves on $X_{an}$ and the small \'etale topos of $X,$ which canonically extends to a geometric morphism of $\i$-topoi
$$\varepsilon:\Shi\left(X_{an}\right) \to \Shi\left(X_{\et}\right).$$
\item[3)] A \emph{$\pi$-finite space} is a space $V$ with only finitely many connected components and only finitely many homotopy groups all of which are finite. The geometric morphism $\varepsilon$ induces a profinite homotopy equivalence if and only if for every $\pi$-finite space $V,$ the induced map between global sections of the constant stack with value $V$ $$\Gamma_{\et}\Delta^{\et}\left(V\right) \to \Gamma_{an}\Delta^{an}\left(V\right)$$ is a homotopy equivalence.
\item[4)] By GAGA, $\varepsilon$ induces an isomorphism on profinite fundamental groups, and by results in \cite{SGA4}, it induces an isomorphism in cohomology with coefficients in any local system of finite abelian groups.
\end{itemize}

In order to deduce that $\varepsilon$ is a profinite homotopy equivalence from $3)$ and $4),$ one needs to understand the interpretation of cohomology classes of a space with coefficients in a local system in terms of classifying spaces, and one needs to know this interpretation is valid in any $\i$-topos. We therefore dedicate Section \ref{sec: cohomology} of this paper to carefully working this out. This allows us to prove the maps in $3)$ are homotopy equivalences by induction using Postnikov towers.

\subsection{Conventions and notation}
By an $\i$-category, we mean an $\left(\i,1\right)$-category. We will model these using quasicategories. We follow very closely the notational conventions and terminology from \cite{htt}, and refer the reader to the index and notational index op. cit. However, we do make a few small deviations from the notational conventions just mentioned:

\begin{itemize}
\item[1.] We shall interchangeably use the notation $\iGpd$ and $\cS$ for the $\i$-category of $\i$-groupoids, or the $\i$-category of spaces, since these are in fact the same $\i$-category. (We find it useful to use one terminology over another in certain instances to emphasize how we are viewing the objects in question.)
\item[2.] For $\sC$ an $\i$-category, we denote by $\Hom_{\sC}\left(C,D\right)$ the space of morphisms from $C$ to $D$ rather than using the notation $\mbox{Map}_{\sC}\left(C,D\right).$
\item[3.] For $\sC$ an $\i$-category, we denote by $\Pshi\left(\sC\right)$ the $\i$-category of $\i$-presheaves, i.e. the functor category $$\Fun\left(\sC^{op},\cS\right)=\Fun\left(\sC^{op},\iGpd\right).$$
\end{itemize}

\subsection*{Acknowledgments}
We would like to thank Anton Fetisov, Geoffroy Horel, Marc Hoyois, Achim Krause, Jacob Lurie, Thomas Nikolaus, Urs Schreiber, Mattia Talpo, and Kirsten Wickelgren for useful conversations, both electronically and in person. We are particularly indebted to Jacob Lurie, who provided key suggestions for the strategy of proof of our comparison theorem for the case of schemes, and to Achim Krause who was nice enough to carefully explain how to interpret twisted cohomology classes in terms of classifying spaces. Finally, we would also like to thank our co-authors of \cite{knhom}: Sarah Scherotzke, Nicol\`o Sibilla, and Mattia Talpo. Had it not been for our project together, we may never have turned our attention towards \'etale homotopy theory.

\newpage

\section{\'Etale Homotopy Theory}
In this section we will present a refinement of the construction of Artin and Mazur for the \'etale homotopy type of a scheme. Our construction is defined for an arbitrary higher stack on the \'etale site, and agrees with the definition of Lurie for Deligne-Mumford stacks.

\subsection{Pro-spaces and profinite spaces}
In this subsection, we give a brief recollection of the concepts of pro-objects, pro-spaces, and profinite spaces. For more detail, we refer the reader to \cite[Section 2]{knhom}.
\begin{definition}
Let $\sC$ be any $\i$-category. Then there is an $\i$-category $\Pro\left(\sC\right)$ together with a fully faithful functor $$j:\sC \hookrightarrow \Pro\left(\sC\right).$$ The $\i$-category $\Pro\left(\sC\right)$ is called the $\i$-category of \textbf{pro-objects} of $\sC,$ and it satisfies the following universal property:\\

$\Pro\left(\sC\right)$ admits small cofiltered limits, and if $\cD$ is any $\i$-category admitting small cofiltered limits, then composition with $j$ induces an equivalence of $\i$-categories

\begin{equation}\label{eq:pro}
\Fun_{\mathit{co-filt.}}\left(\Pro\left(\sE\right),\sD\right) \to \Fun\left(\sE,\sD\right),
\end{equation}
where $\Fun_{\mathit{co-filt.}}\left(\Pro\left(\sE\right),\sD\right)$ is the full subcategory of $\Fun\left(\Pro\left(\sE\right),\sD\right)$ spanned by those functors which preserve small cofiltered limits.
\end{definition}

In many cases, the $\i$-category $\Pro\left(\sC\right)$ can be described explicitly. When $\sC$ is small, then we can identify $\Pro\left(\sC\right)$ with the full subcategory of $\Fun\left(\sC,\cS\right)^{op}$ spanned by those functors which are cofiltered limits of co-representable functors (those of the form $\Hom_{\sC}\left(C,\blank\right),$ for $C$ an object of $\sC$). In this case, the functor $j$ is simply the Yoneda embedding (of $\sC^{op}$). In fact, this description persists for $\sC$ a large (but locally small) $\i$-category, provided we replace $\cS$ with the $\i$-category of large spaces, $\widehat{\cS},$ and we demand that the cofiltered limits we are considering are small. However, if $\sC$ is accessible and admits finite limits, then there is a more concrete description of $\Pro\left(\sC\right),$ namely it is the full subcategory of $\Fun\left(\sC,\cS\right)^{op}$ on those functors which are left exact and accessible \cite[Proposition 3.1.6]{dagxiii}.

\begin{remark}
In all the cases above, the functor $$j:\sC \to \Pro\left(\sC\right)$$ can be identified with a restriction of the opposite functor of the Yoneda embedding $$y:\sC^{op} \hookrightarrow \Pshi\left(\sC^{op}\right)=\Fun\left(\sC,\cS\right),$$ and since $y$ is fully faithful and preserves limits, $j$ is fully faithful and preserve colimits.
\end{remark}

\begin{definition}
The $\i$-category $\Pro\left(\cS\right)$ is the $\i$-category of \textbf{pro-spaces}.
\end{definition}

\begin{definition}\label{dfn:pi-finite}
A space $X$ in $\cS$ is \textbf{$\pi$-finite} if all its homotopy groups are finite, it has only finitely many non-trivial homotopy groups, and finitely many connected components.
\end{definition}

\begin{definition}
Let $\cS^{\pi}$ denote the full subcategory of the $\i$-category $\cS$ on the $\pi$-finite spaces. $\cS^{\pi}$ is essentially small and idempotent complete (and hence accessible). The $\i$-category of \textbf{profinite spaces} is defined to be the $\i$-category $$\Profs := \Pro\left(\cS^{\pi}\right).$$
\end{definition}

Denote by $i$ the canonical inclusion $i:\cS^{\pi} \hookrightarrow \cS.$
It induces a fully faithful embedding 
$$\Pro\left(i\right):\Prof\left(\cS\right) \hookrightarrow \Pro\left(\cS\right)$$
of profinite spaces into pro-spaces \cite[Remark 3.1.7]{dagxiii}. It is the functor corresponding under (\ref{eq:pro}) with the composite 
$$\cS^{\pi} \stackrel{i}{\hookrightarrow} \cS \stackrel{j}{\hookrightarrow} \Pro\left(\cS\right).$$
Moreover, $i$ is accessible and preserves finite limits, hence the above functor has a left adjoint $$i^*:\Pro\left(\cS\right) \to \Prof\left(\cS\right)$$ induced by composition with $i,$ by loc. cit.

\begin{definition}
We denote by $\widehat{\left(\blank\right)}$ the composite
$$\cS \stackrel{j}{\hookrightarrow} \Pro\left(\cS\right) \stackrel{i^*}{\longrightarrow} \Prof\left(\cS\right)$$ and call it the \textbf{profinite completion functor}. Concretely, if
$X$ is a space in $\cS,$ then $\widehat{X}$ corresponds to the composite $$\cS^\pi \stackrel{i}{\hookrightarrow} \cS \stackrel{\Hom\left(X,\blank\right)}{\longlonglongrightarrow} \cS.$$ This functor has a right adjoint given by the composite $$\Prof\left(\cS\right) \stackrel{\Pro\left(i\right)}{\longlonghookrightarrow} \Pro\left(\cS\right) \stackrel{T}{\longrightarrow} \cS,$$ where $T$ sends a functor $$F:\cS \to \cS$$ corresponding to a pro-space to $F\left(*\right),$ that is, the evaluation of $F$ on the one-point space \cite[Proposition 2.8]{knhom}. Concretely, $T$ sends a pro-space of the form $\underset{i} \lim j\left(X_i\right)$ to the actual limit in $\cS$ $$\underset{i} \lim X_i,$$ see \cite[Proposition 2.10]{knhom}.
\end{definition}

\subsection{Profinite shape theory}

We first begin by recalling how to associate to a space $X$ in $\cS,$ an $\i$-topos. To do this, it is conceptually simpler to view such an object $X$ as an $\i$-groupoid, as then we have a natural candidate for such an $\i$-topos, namely the $\i$-topos of $\i$-presheaves on $X,$ $\Pshi\left(X\right).$ Thinking more topos-theoretically, viewing $X$ as an object of the $\i$-topos $\cS$ of spaces, another natural candidate for such an $\i$-topos is the slice $\i$-topos $\cS/X,$ and these two natural choices agree by \cite[Corollary 5.3.5.4]{htt}. By \cite[Remark 6.3.5.10, Theorem 6.3.5.13, and Proposition 6.3.4.1]{htt}, it follows that there is a fully faithful colimit preserving functor
\begin{eqnarray*}
\cS/\left(\mspace{3mu}\cdot\mspace{3mu}\right): \cS &\to& \mathfrak{Top}_\i\\
X &\mapsto& \cS/X
\end{eqnarray*}
from the $\i$-category of spaces to the $\i$-category of $\i$-topoi.

\begin{remark}
The above functor is not to be confused with the functor
\begin{eqnarray*}
\Sh_\i\left(\mspace{3mu}\cdot\mspace{3mu}\right):\Top &\to& \mathfrak{Top}_\i\\
T &\mapsto& \Shi\left(T\right)
\end{eqnarray*}
sending a \emph{topological} space $T$ to its $\i$-topos of $\i$-sheaves. The above functor however is also fully faithful, once one restricts it to the full subcategory of sober topological spaces. If $T$ is a (sober) topological space and $\Pi_\i T$ is its associated $\i$-groupoid, then $\Shi\left(T\right)$ remembers the space $T$ up to homeomorphism, whereas $\cS/\left(\Pi_\i T\right)$ only captures the weak homotopy type of $T$. For nice spaces, one can recover $\cS/\left(\Pi_\i T\right)$ however as the \emph{shape} of $\Shi\left(T\right)$ (or its hypercompletion), see Proposition \ref{prop:shape1} and Proposition \ref{prop: locally contractible shape}.
\end{remark}

The functor $$\cS/\left(\mspace{3mu}\cdot\mspace{3mu}\right): \cS \to \mathfrak{Top}_\i,$$ by the equivalence (\ref{eq:pro}), induces a well-defined functor $$\Pro\left(\cS\right) \to \mathfrak{Top}_\i$$ which sends a representable pro-space $j\left(X\right)$ to $\cS/X,$ and sends a pro-space of the form $\underset{ i \in \cI} \lim X_i$ to the cofiltered limit of $\i$-topoi $\underset{ i \in \cI} \lim \cS/X_i.$ Denote this functor by $\cS^{pro}/\left(\mspace{3mu}\cdot\mspace{3mu}\right).$ By \cite[Remark 7.1.6.15]{htt}, this functor has a left adjoint $\Shape.$ We now will describe this construction, which originates from \cite{ToVe}:

Recall that a morphism in $\mathfrak{Top}_\i$ $$f:\cE \to \cF,$$ called a \emph{geometric morphism}, consists of an adjunction $f^* \dashv f_*,$ such that the left adjoint $f^*$ preserves finite limits. Let $\cE$ be an $\i$-topos. Consider the essentially unique geometric morphism $e:\cE \to \cS$ to the terminal $\i$-topos of spaces. Then the composite $$\cS \stackrel{e^*}{\longrightarrow} \cS \stackrel{e_*}{\longrightarrow} \cS$$ is a left-exact functor, i.e. a pro-space. We typically denote the inverse image functor $e^*$ as $\Delta$ and the direct image functor $e_*$ as $\Gamma.$ The functor $\Gamma$ is the global sections functor, i.e. it sends an object $E$ to the space $\Hom_{\cE}\left(1,E\right).$
We denote by $\Shape\left(\cE\right)$ the pro-space $\Gamma\circ \Delta.$

\begin{definition}\label{dfn:shape}
Let $\cE$ be an $\i$-topos. Then the pro-space $\Shape\left(\cE\right)$ is called the \textbf{shape} of the $\i$-topos $\cE.$
\end{definition}

We have the following useful proposition

\begin{proposition} (\cite[Remark A.1.4]{higheralgebra})\label{prop:shape1}
Let $T$ be a paracompact Hausdorff space homotopy equivalent to a $CW$-complex. Then $$\Shape\left(\Shi\left(T\right)\right) \simeq j\left(\Pi_\i T\right).$$
\end{proposition}

\begin{definition}
An $\i$-topos $\cE$ is \textbf{locally $\i$-connected} if the inverse image functor $$\Delta:\cS \to \cE$$ has a left adjoint $\Pi^{\cE}_\i.$
\end{definition}

\begin{remark}\label{rmk: locally infinity connected}
If $\cE$ is a locally $\i$-connected $\i$-topos, then the pro-space $\Shape\left(\cE\right)$ is corepresented by the space $$\Pi^{\cE}_\i\left(1\right).$$ This follows from the fact that if $\cG$ is any space in $\cS,$ by adjunction we have the following natural equivalences
\begin{eqnarray*}
\Hom_{\cS}\left(\Pi^{\cE}_\i\left(1\right),\cG\right) &\simeq& \Hom_{\cE}\left(1,\Delta\left(\cG\right)\right)\\
&=& \Gamma\Delta\left(\cG\right)\\
&=& \Shape\left(\cE\right)\left(\cG\right).
\end{eqnarray*}
\end{remark}

We also have the following proposition:

\begin{proposition}\label{prop: locally contractible shape}
Let $T$ be a locally contractible topological space. Then the shape $$\Shape\left(\Hshi\left(T\right)\right)$$ of its $\i$-topos of hypersheaves is equivalent to $j\left(\Pi_\i T\right).$ Moreover, $\Hshi\left(T\right)$ is locally $\i$-connected.
\end{proposition}

\begin{proof}
Denote by $\Pi_\i:\Top \to \cS$ the canonical functor sending a space to its weak homotopy type (as an $\i$-groupoid). Denote by $Op\left(T\right)$ the poset of open subsets of $T.$ Denote by $l$ the composite
$$Op\left(T\right) \to \Top \stackrel{\Pi_\i}{\longlongrightarrow} \cS,$$ where the functor $Op\left(T\right) \to \Top$ sends each open subset $U$ of $T$ to itself. Denote by $$L=\Lan_y l:\Pshi\left(Op\left(T\right)\right) \to \cS$$ the left Kan extension of $l$ along the Yoneda embedding, i.e. the unique colimit preserving functor which agrees with $l$ on representables. It follows from the Yoneda lemma that this functor has a right adjoint $R$ which sends an $\i$-groupoid $\cG$ to the $\i$-presheaf $$R\left(\cG\right):U \mapsto \Hom\left(l\left(U\right),\cG\right).$$ We claim that $R\left(\cG\right)$ is a hypersheaf. To see this, it suffices to observe that if $V^{\bullet}$ is a hypercover of $U$, then, regarding it in the natural way as a simplicial topological space, the colimit of the composite $$\Delta^{op} \stackrel{V^\bullet}{\longlongrightarrow} \Top \stackrel{\Pi_\i}{\longlongrightarrow} \cS$$ is $l\left(U\right),$ which follows from \cite[Theorem 1.3]{duggerisaksen}. It follows that $R$ and $L$ restrict to adjoint functors
$$\xymatrix@C=2cm{\Hshi\left(T\right) \ar@<-0.65ex>[r]_-{L} & \cS. \ar@<-0.65ex>[l]_-{R}}$$ Denote by $Op^c\left(T\right)$ the subposet of $Op\left(T\right)$ on those open subsets which are contractible. Then, since $T$ is locally contractible, by the Comparison Lemma \cite[III]{SGA4}, it follows that $$\Sh\left(Op\left(T\right)\right)\simeq \Sh\left(Op^c\left(T\right)\right),$$ where the latter topos is the topos of sheaves with respect to covers by contractible open subsets. It now follows from \cite[Theorem 5]{Jardine} and \cite[Proposition 6.5.2.14]{htt} that there is a canonical equivalence $$\Hshi\left(Op\left(T\right)\right)\simeq \Hshi\left(Op^c\left(T\right)\right).$$ The left adjoint $\Delta$ to global sections in $\Hshi\left(Op^c\left(T\right)\right)$ is defined so that $\Delta\left(\cG\right)$ can be computed as the hypersheafification of constant presheaf with value $\cG.$ Note however that for $U$ in $Op^c\left(T\right),$ $R\left(\cG\right)$ is a hypersheaf, and
\begin{eqnarray*}
R\left(\cG\right)\left(U\right) &\simeq& \Hom\left(l\left(U\right),\cG\right)\\
&\simeq& \Hom\left(*,\cG\right)\\
&\simeq& \cG,
\end{eqnarray*}
since $U$ is contractible, and hence the constant presheaf is already a hypersheaf on $Op^c\left(T\right).$ Hence we can identify $R$ with $\Delta.$ It follows that $\Hshi\left(T\right)$ is locally $\i$-connected with $$\Pi^{T}_\i=L \dashv \Delta.$$ By Remark \ref{rmk: locally infinity connected}, it follows that the shape of $\Hshi\left(T\right)$ is corepresented by $j\left(\Pi^{T}_\i\left(1\right)\right).$ But $1$ is the representable presheaf corresponding to the open subset $T,$ and hence $\Pi^{T}_\i\left(1\right)$ is canonically equivalent to $$l\left(T\right)=\Pi_\i\left(T\right).$$
\end{proof}

Consider the profinite completion functor from pro-space to profinite spaces $$i^*:\Pro\left(\cS\right) \to \Profs.$$
By composition we get a functor
$$\mathfrak{Top}_\i \stackrel{\Shape}{\longlongrightarrow} \Pro\left(\cS\right) \stackrel{i^*}{\longrightarrow} \Profs,$$ which we shall denote by $\Shape^{\Prof}$, whose right adjoint is given by the composition
$$\Profs \stackrel{\Pro\left(i\right)}{\longlongrightarrow} \Pro\left(\cS\right) \stackrel{\cS^{\Pro}/\left(\mspace{3mu}\cdot\mspace{3mu}\right)}{\longlongrightarrow} \mathfrak{Top}_\i,$$
which we shall denote by $\cS^{\Prof}/\left(\mspace{3mu}\cdot\mspace{3mu}\right).$

\begin{definition}\label{dfn:profiniteshape}
Let $\cE$ be an $\i$-topos. Then the profinite space $\Shape^{\Prof}\left(\cE\right)$ is called the \textbf{profinite shape} of the $\i$-topos $\cE.$
\end{definition}

\begin{definition}\label{dfn:profinite eql}
Let $\cE \to \cF$ be a geometric morphism of $\i$-topoi. Such a morphism is a \textbf{profinite homotopy equivalence} if the induced map $$\Shape^{\Prof}\left(\cE\right) \to \Shape^{\Prof}\left(\cF\right)$$ is an equivalence of profinite spaces.
\end{definition}

\begin{remark}\label{rmk:concprof}
Unraveling the definitions, we see that a geometric morphism $\varphi:\cE \to \cF$ is a profinite homotopy equivalence, if and only if for every $\pi$-finite space $V,$ the canonical morphism
$$f_*f^*\left(V\right)\to e_*e^*\left(V\right)$$ is an equivalence of spaces,
where $e:\cE \to \cS$ and $f:\cF \to \cS$ are the (essentially unique) maps to the terminal $\i$-topos, and the above map is induced by the unit of the adjunction $\varphi^* \dashv \varphi_*,$ using the equivalence $e \simeq f \circ \varphi.$
\end{remark}

\begin{proposition}\label{prop:hypersame}
Let $\cE$ be an $\i$-topos and let $a:\widehat{\cE} \to \cE$ be the canonical map from its hypercompletion. Then $a$ is a profinite homotopy equivalence.
\end{proposition}

\begin{proof}
Using Remark \ref{rmk:concprof}, it suffices to show for any $\pi$-finite space $V$ that the canonical map $$e_*e^*\left(V\right) \to e_*\left(a_*a^*\left(e^*\left(V\right)\right)\right)$$ is an equivalence, where the map above is induced from the map $$e^*\left(V\right) \to a_*a^*e^*\left(V\right).$$ However, the latter map is the canonical map from $e^*\left(V\right)$ to its hypersheafification. Since $V$ is $\pi$-finite, it is $n$-truncated for some $n,$ and hence so is $e^*\left(V\right)$ by \cite[Proposition 5.5.6.16]{htt}, since $e^*$ is left exact. However, every $n$-truncated sheaf automatically satisfies hyperdescent, so this map is an equivalence.
\end{proof}

\subsection{The small \'etale $\i$-topos.}

For this subsection and the next, we will work over an arbitrary commutative ring $k$.

\begin{definition}\label{dfn:etaleclosed}
A subcategory $\Affku$ of the category of affine schemes over $k,$ $\mathbf{Aff}_k,$ is \textbf{\'etale closed} if for any object $X$ in $\Affku,$ if $Y \to X$ is an \'etale morphism from an affine scheme, then $Y$ is in $\Affku$. We denote the corresponding subcategory of commutative $k$-algebras by $\Algku.$
\end{definition}

\begin{example}
Since \'etale maps are of finite presentation, the subcategory $\Affk$ of affine $k$-schemes of finite type is \'etale closed.
\end{example}

\begin{remark}
Since \'etale maps between affine schemes are of finite presentation, any essentially small subcategory of $\mathbf{Aff}_k$ is contained in an essentially small \'etale closed subcategory.
\end{remark}

In the rest of this section, we will work over a fixed essentially small \'etale closed subcategory $\Affku$ of $\mathbf{Aff}_k,$ unless otherwise specified. The notion of \'etale closedeness was specifically chosen so that the \'etale pretopology on $\mathbf{Aff}_k$ naturally restricts.

Recall that for a scheme $X,$ its \textbf{small \'etale site} is the following Grothendieck site: As a category, $X_{\et}$ consists of \'etale morphisms $$U \to X,$$ with $U$ another scheme, and the morphisms are commutative triangles over $X.$ 
The Grothendieck pretopology on this category is given by \'etale covering families.
The small \'etale topos $\Sh\left(X_{\et}\right)$ is the topos of sheaves over this site. 

\begin{definition}
Let $X$ be a scheme. Its \textbf{small \'etale $\i$-topos} is the $\i$-topos $\Shi\left(X_{\et}\right).$
\end{definition}

\begin{remark}
Since \'etale maps are stable under pullback, the category $X_{\et}$ has finite limits. It follows then from \cite[Lemma 6.4.5.6]{htt} that $\Shi\left(X_{\et}\right)$ is the $1$-localic $\i$-topos corresponding to $\Sh\left(X_{\et}\right)$ under the equivalence of $\i$-categories between the $\left(2,1\right)$-category of topoi and the $\i$-category of $1$-localic $\i$-topoi.
\end{remark}

This definition naturally carries over for Deligne-Mumford stacks and their higher analogues. It will be technically convenient to work straightaway with higher Deligne-Mumford stacks. We start by briefly recalling some material from \cite{dagv} and \cite{higherdave}.

Let $A$ be a commutative $k$-algebra, where $k$ is our base ring. Denote by $A_{\et}$ the category whose objects consists of \'etale morphisms $U \to \Spec\left(A\right),$ with $U$ another \emph{affine} scheme. There is a canonical inclusion of sites $$A_{\et} \hookrightarrow \Spec\left(A\right)_{\et},$$ which satisfies the conditions of the Comparison Lemma of \cite{SGA4} III, hence one has $$\Sh\left(A_{\et}\right)\simeq \Sh\left(\Spec\left(A\right)_{\et}\right).$$ As both sites have finite limits, it follows from \cite[Proposition 6.4.5.4]{htt} that $$\Shi\left(A_{\et}\right) \simeq \Shi\left(\Spec\left(A\right)_{\et}\right).$$ Notice that there is a canonical sheaf of rings $\mathcal{O}_A$ on the site $A_{\et},$ which assigns an \'etale map $$\Spec\left(B\right) \to \Spec\left(A\right)$$ the ring $B.$ The stalks of this sheaf $\mathcal{O}_{A}$ along geometric points are not only local $k$-algebras, but in fact \emph{strictly Henselian}. This is important in order to get the correct notion of morphism between ringed topoi. Just as a map of ringed spaces need not be a map of \emph{locally} ringed spaces, since one must demand that the induced map along stalks is a map of local rings, i.e. preserves the maximal ideals, a map of strictly Henselian ringed topoi needs to respect the Henselian ring structure along stalks, i.e. be a Henselian map. Using this idea one can define an $\i$-category of strictly Henselian ringed $\i$-topoi. Let us denote this $\i$-category by $\Htop$. (To be precise, $\Htop$ is $\i$-category ${\mathcal L}{\mathcal T}op\left(\cG\right)^{op}$ defined in \cite[Definition 1.4.8]{dagv}, with $\cG$ the \'etale geometry in the sense of Section 2.6 of op. cit.) By \cite[Theorem 2.2.12]{dagv} with $\mathcal{G}$ the \'etale geometry as in Section 2.6 of op. cit., the construction $$A \mapsto \Shi\left(A_{\et}\right)$$ can be turned into a fully faithful functor
$$\Spec_{\et}:\Affku \hookrightarrow \Htop$$ from affine $k$-schemes of finite type over $k$ to $\i$-topoi locally ringed in strict Henselian rings.

The $\i$-category $\Htop$ carries a natural Grothendieck topology \cite[Definition 4.3.2]{higherdave}, also called the \emph{\'etale topology}, which is a natural extension of the classical \'etale topology on $\Affku$ with respect to the functor $\Spec_{\et}.$ We say that a strictly Henselian ringed $\i$-topos $\left(\cE,\cO_{\cE}\right)$ is ($\cU$-)\textbf{Deligne-Mumford} if there exists an \'etale covering family $$\left(\left(\cE_\alpha,\cO_{E_\alpha}\right) \to \left(\cE,\cO_{E}\right)\right)_\alpha$$ such that for each $\alpha,$ $$\left(\cE_\alpha,\cO_{E_\alpha}\right) \simeq \Spec_{\et}\left(A_\alpha\right),$$ for $A_\alpha \in \Algku.$ We will call a Deligne-Mumford strictly Henselian ringed $\i$-topos a \textbf{Deligne-Mumford scheme}, as these are precisely $\cG$-schemes in the sense of \cite[Definition 2.3.9]{dagv}, where $\cG$ is the \'etale geometry in the sense of Section 2.6 of op. cit. (and the \'etale cover is by affines in $\Affku$). We denote the $\i$-category of Deligne-Mumford schemes by $\Dmsch.$

Restriction along $\Spec_{\et}$ defines for each strictly Henselian ringed $\i$-topos $\left(\cE,\cO_{E}\right)$ a functor of points
\begin{eqnarray*}
\tilde y\left(\left(\cE,\cO_{E}\right)\right):\left(\Affku\right)^{op} &\to& \iGpd\\
\Spec\left(A\right) &\mapsto& \Hom_{\Htop}\left(\Spec_{\et}\left(A\right),\left(\cE,\cO_{E}\right)\right).
\end{eqnarray*}
Each functor $\tilde y\left(\left(\cE,\cO_{E}\right)\right)$ satisfies \'etale descent and the restriction of $\tilde y$ to Deligne-Mumford schemes defines a fully faithful functor $$\tilde y:\Dmsch \hookrightarrow \Shi\left(\Affku,\et\right)$$ \cite[Theorem 2.4.1, Lemma 2.4.13]{dagv}, \cite[Theorem 5.2.2, Remark 5.2.3]{higherdave}.

\begin{definition}
A \textbf{Deligne-Mumford $\i$-stack} is an $\i$-stack $\cX$ on the \'etale site of $\Affku,$ equivalent to the functor of points of a Deligne-Mumford scheme. We denote the $\i$-category of such stacks by $\mathfrak{DM}\left(k\right)_\i^{\cU}.$
\end{definition}

\begin{remark}
By \cite[Theorem 2.6.18]{dagv}, $\mathfrak{DM}\left(k\right)_\i^{\cU}$ contains the classical $\left(2,1\right)$-category of Deligne-Mumford stacks that can be modeled on affines in $\Affku$ as a full subcategory, but also contains more general objects as there are no separation conditions imposed. E.g., $B\left(\mathbb{Z}\right)$ we will be considered a Deligne-Mumford stack in this setting. This will cause no problems and will in fact simplify the proofs considerably.
\end{remark}

\begin{definition}\label{dfn:smalletaledm}
Let $\cX$ be a Deligne-Mumford ($\i$-)stack. Its \textbf{small \'etale $\i$-topos} is the $\i$-topos of $\i$-sheaves over $\cX_{\et},$ $\Shi\left(\cX_{\et}\right),$ where $\cX_{\et}$ is the ($\i$-)category of (not necessarily representable) \'etale maps $U \to \cX$ with $U$ a scheme, equipped with the Grothendieck topology generated by \'etale covering families. If $\cX\simeq \tilde y\left(\left(\cE,\cO_{\cE}\right)\right)$, by an \'etale morphism, we mean a morphism $U \to \cX$ which under the Yoneda lemma corresponds to a morphism $$\left(\Shi\left(U_{\et}\right),\cO_{U}\right) \to \left(\cE,\cO_{\cE}\right)$$ of Deligne-Mumford schemes which is \'etale in the sense of \cite[Definition 2.3.1]{dagv}.
\end{definition}

\begin{lemma}\label{lem:sametopos}
Let $\cX$ be a Deligne-Mumford $\i$-stack. By definition, $\cX$ is the functor of points of a Deligne-Mumford scheme $\left(\cE,\mathcal{O}_\cE\right)$. In this case, one has a canonical equivalence $$\Shi\left(\cX_{\et}\right) \simeq \cE.$$
\end{lemma}

\begin{proof}
Let $\mathfrak{DM}\left(k\right)_\i^{\cU}$ denote the $\i$-category of Deligne-Mumford $\i$-stacks built out of affine schemes in $\Affku$, and define similarly $\Sch^{\cU}_k$ to be the analogously defined category of schemes.
By \cite[Remarks 5.19 and 5.23]{higherdave}, we see that $\Sch^{\cU}_k$ is a locally small strong \'etale blossom in the sense of \cite[Definition 5.1.7]{higherdave}, or more precisely, $\Sch^{\cU}_k$ is canonically equivalent to the strong \'etale blossom whose objects are the Deligne-Mumford schemes which are classical schemes built out of affines in $\Affku.$ By \cite[Theorems 5.3.6 and 5.3.7]{higherdave} combined with Proposition 5.3.2 of op. cit., it follows in fact that $$\cE\simeq \Shi\left(\Sch^{{\cU,\et}}_k,\right)/\tilde y^{\et}\left(\cX\right),$$ where $\tilde y^{\et}\left(\cX\right)$ is the stack assigning a scheme $X$ the $\i$-groupoid of \'etale maps $$X \to \cX.$$ By \cite[Remark 2.2.4 and Proposition 2.2.1]{higherdave}, we conclude that $\cE\simeq\Shi\left(\cX_{\et}\right).$
\end{proof}

\begin{lemma}\label{lem:smalletale}
There is a colimit preserving functor $$\Shi\left(\left(\blank\right)_{\et}\right):\Shi\left(\Affku,\et\right) \to \Topi$$ which sends an affine scheme $\Spec\left(A\right)$ to $\Shi\left(A_{\et}\right).$
\end{lemma}

\begin{proof}
We start by constructing such a functor out of $\i$-presheaves, which can be accomplished simply by taking the left Kan extension of the functor $$\Affku \stackrel{\Spec_{\et}}{\longlonglongrightarrow} \Htop \to \mathfrak{Top}_\i$$ along the Yoneda embedding $$y:\Affku \hookrightarrow \Pshi\left(\Affku\right),$$ where $$\Htop \to \mathfrak{Top}_\i$$ is the canonical functor which forgets the structure sheaf. Denote this left Kan extension by $L.$ By \cite[Proposition 5.5.4.20 and Theorem 5.1.5.6]{htt}, it suffices to show that $L$ sends each covering sieve $$S_\mathcal{U} \hookrightarrow y\left(\Spec\left(A\right)\right)$$ for $\mathcal{U}=\left(U_i \to \Spec\left(A\right)\right)_i$ an \'etale covering family, to an equivalence. Note however that this covering sieve is the colimit of the \v{C}ech nerve $$N_\mathcal{U}:\Delta^{op} \to \Pshi\left(\Affku\right)$$ of $\mathcal{U}.$ Since $L$ preserves colimits, it thus suffices to show that the canonical map $$\colim L\circ N_{\mathcal{U}} \to L\left(y\left(\Spec\left(A\right)\right)\right)\simeq \Shi\left(A_{\et}\right)$$ is an equivalence. 

The functor $N_\mathcal{U}$ has a canonical lift to an augmented simplicial diagram $$\widehat{N_\mathcal{U}}:\left(\Delta^{op}\right)^{\triangleright}\cong \Delta_{+}^{op} \longlongrightarrow \Pshi\left(\Affku\right)$$ defining the canonical cocone for $N_\mathcal{U}$ with vertex $y\left(\Spec\left(A\right)\right)$ (which corresponds to the inclusion of the subobject $S_\mathcal{U} \hookrightarrow y\left(\Spec\left(A\right)\right)$). At the level of simplicial sets, the formation of right cones is strictly left adjoint to the formation of slice quasicategories, so the map of simplicial sets $\widehat{N_\mathcal{U}}$ is adjoint to a map $$\widetilde{N}_\mathcal{U}:\Delta^{op} \to \Pshi\left(\Affku\right)/y\left(\Spec\left(A\right)\right).$$ Now $L$ induces a colimit preserving functor $$\widetilde{L}:\Pshi\left(\Affku\right)/y\left(\Spec\left(A\right)\right) \to \mathfrak{Top}_\i/\Shi\left(A_{\et}\right)$$ By \cite[Example 2.3.8]{dagv} together with the fact that along representables $L$ agrees with $\i$-sheaves on the small \'etale site, the diagram $\widetilde{L}\circ \widetilde{N}_{\mathcal{U}}$ consists of \'etale geometric morphisms of $\i$-topoi over $\Shi\left(A_{\et}\right)$ and therefore there is a factorization of $\widetilde{L}\circ \widetilde{N}_{\mathcal{U}}$ of the form $$\Delta^{op} \stackrel{F}{\longrightarrow} \mathfrak{Top}^{\et}/\Shi\left(A_{\et}\right) \to \mathfrak{Top}/\Shi\left(A_{\et}\right),$$ where $ \mathfrak{Top}^{\et}$ denotes the $\i$-category of $\i$-topoi and \'etale geometric morphisms, and moreover the composite $$\Delta^{op} \stackrel{F}{\longrightarrow} \mathfrak{Top}^{\mbox{\'et}}/\Shi\left(A_{\et}\right) \to \mathfrak{Top}^{\et}_\i \to \mathfrak{Top}_\i$$ agrees up to equivalence with $L \circ N_{\cU}.$ Note that by \cite[Remark 6.3.5.10]{htt}, there is a canonical equivalence of $\i$-categories $\mathfrak{Top}^{\et}/\Shi\left(A_{\et}\right) \simeq \Shi\left(A_{\et}\right)$ under which $F$ corresponds to the \v{C}ech nerve of the same \'etale cover, except regarded as a simplicial diagram $$\overline{N}_{\cU}:\Delta^{op} \to \Shi\left(A_{\et}\right).$$ The colimit of this diagram is the terminal object. Notice that $\mathfrak{Top}^{\et}/\Shi\left(A_{\et}\right) \to \mathfrak{Top}^{\et}_\i$ preserves colimits and so does $\mathfrak{Top}^{\et}_\i \to \mathfrak{Top}_\i$ by \cite[Theorem 6.3.5.13]{htt}. The result now follows since the terminal object gets sent to $\Shi\left(A_{\et}\right)$ under the composite $$\mathfrak{Top}^{\et}/\Shi\left(A_{\et}\right) \to \mathfrak{Top}^{\et}_\i \to \Topi.$$
\end{proof}

\begin{definition}
Let $F$ be an $\i$-sheaf on $\left(\Affku,\et\right).$ Then the \textbf{small \'etale $\i$-topos of $F$} is $$\Shi\left(F_{\et}\right):=\Shi\left(\left(\blank\right)_{\et}\right)\left(F\right).$$
\end{definition}

We will proceed to justify this definition by showing it agrees with Definition \ref{dfn:smalletaledm} when $F$ is a Deligne-Mumford ($\i$-)stack. First, we will show that the definition does not depend on the ambient \'etale closed subcategory:

\begin{remark}
Suppose that $\mathbf{Aff}^{\cV}_k$ is an essentially small \'etale closed category of affine schemes which contains $\Affku$. Denote by $i$ the inclusion $$i:\Affku \hookrightarrow \mathbf{Aff}^{\cV}_k.$$ Then $i$ induces a restriction functor $$i^*\Shi\left(\mathbf{Aff}^{\cV}_k,\et\right) \to \Shi\left(\Affku,\et\right)$$ which has a fully faithful left adjoint $$i_!:\Shi\left(\Affku,\et\right) \hookrightarrow \Shi\left(\mathbf{Aff}^{\cV}_k,\et\right).$$ Concretely, $i_!$ is the unique colimit preserving functor sending each affine scheme $\Spec\left(A\right)$ to itself. Denote by $$S:\Shi\left(\Affku,\et\right) \to \Topi$$ the colimit preserving functor from Lemma \ref{lem:smalletale}, and similarly denote by $$R:\Shi\left(\mathbf{Aff}^{\cV}_k,\et\right) \to \Topi$$ the corresponding functor for the \'etale closed category $\mathbf{Aff}^{\cV}_k.$ Then we have a canonical equivalence $$S\left(F\right)\simeq R\left(i_!\left(F\right)\right).$$ To see this, note the composition $$\Shi\left(\Affku,\et\right) \stackrel{i_!}{\longrightarrow} \Shi\left(\mathbf{Aff}^{\cV}_k,\et\right) \stackrel{R}{\longrightarrow} \Topi$$ is colimit preserving and agrees with $S$ along representables. It follows that the above composition $$R \circ i_! \simeq S.$$
\end{remark}

We will now justify the notation for the functor $\Shi\left(\left(\blank\right)_{\et}\right)$ by showing it agrees with Definition \ref{dfn:smalletaledm}:

\begin{proposition}
Let $S=\Shi\left(\left(\blank\right)_{\et}\right)$ be the functor from Lemma \ref{lem:smalletale}, and let $\cX$ be a Deligne-Mumford $\i$-stack. Then $S\left(\cX\right)$ is equivalent to the small \'etale $\i$-topos of $\cX,$ in the sense of Definition \ref{dfn:smalletaledm}.
\end{proposition}

\begin{proof}
Denote by $\mathfrak{DM}\left(k\right)_\i^{\cU,\et}$ the subcategory of Deligne-Mumford $\i$-stacks built out of affines in $\cU,$ where the morphisms are (not necessarily representable) \'etale maps. By \cite[Proposition 5.2.11 and Remark 5.3.11]{higherdave}, it follows that the composition
$$\mathfrak{DM}\left(k\right)_\i^{\cU,\et} \to \mathfrak{DM}\left(k\right)_\i^{\cU} \hookrightarrow \Shi\left(\Affku,\et\right)$$ preserves colimits. 
Hence $$\mathfrak{DM}\left(k\right)_\i^{\cU,\et} \to \mathfrak{DM}\left(k\right)_\i^{\cU} \hookrightarrow \Shi\left(\Affku,\et\right) \stackrel{S}{\longrightarrow} \Topi$$ preserves colimits as well. By \cite[Remark 5.3.11, Lemma 5.1.1, and Proposition 4.3.1]{dagv}, so does the composite $$\mathfrak{DM}\left(k\right)_\i^{\cU,\et} \simeq \Dmschet \to \Topi^{\et},$$ where $\Dmschet$ is the $\i$-category of Deligne-Mumford schemes and their \'etale morphisms, and hence, by \cite[Theorem 6.3.5.13]{htt}, the composite $$\mathfrak{DM}\left(k\right)_\i^{\cU,\et} \simeq \Dmschet \to \Topi^{\et} \to \Topi$$ also preserves colimits. Note that the latter composite sends a Deligne-Mumford $\i$-stack $\cX$ which is the functor of points of a Deligne-Mumford scheme $\left(\cE,\mathcal{O}_\cE\right)$ to the $\i$-topos $\cE.$ So both composites $$\mathfrak{DM}\left(k\right)_\i^{\cU,\et} \simeq \Dmschet \to \Topi^{\et} \to \Topi$$ and $$\mathfrak{DM}\left(k\right)_\i^{\cU,\et} \to \mathfrak{DM}\left(k\right)_\i^{\cU} \hookrightarrow \Shi\left(\Affku,\et\right) \stackrel{S}{\longrightarrow} \Topi$$ are colimit preserving and send an affine scheme $\Spec\left(A\right)$ to $\Shi\left(A_{\et}\right).$ By \cite[Theorem 5.37]{higherdave} (combined with Remark 4.31 of op. cit.) we see that there is a canonical equivalence of $\i$-categories $$\Shi\left(\mathbf{Aff}^{\cU,\et}\right) \simeq \mathfrak{DM}\left(k\right)_\i^{\cU,\et}$$ which sends the sheaf of \'etale points of an affine scheme $\Spec\left(A\right)$ to $\Spec\left(A\right)$ itself. The result now follows from \cite[Proposition 5.5.4.20 and Theorem 5.1.5.6]{htt}, together with Lemma \ref{lem:sametopos} of this article.
\end{proof}

\begin{remark}
The results of this section readily generalize to the settings of derived and spectral algebraic geometry. The proofs are exactly the same, once one replaces the \'etale geometry (\cite[Definition 2.6.12]{dagv}) with the derived \'etale geometry (\cite[Definition 4.3.13]{dagv}) or the spectral \'etale geometry  (\cite[Definition 8.11]{dagvii}) respectively. We presented the results in the setting of non-derived schemes merely to avoid overburdening the reader with new concepts.
\end{remark}

\subsection{\'Etale homotopy type}

We now present the definition of the \'etale homotopy type of a general $\i$-sheaf on the \'etale site of $\Affku$, for some small \'etale closed subcategory of affine $k$-schemes, in the sense of Definition \ref{dfn:etaleclosed}.

\begin{definition}\label{dfn:etalehomotopytype}
The \textbf{\'etale fundamental $\i$-groupoid} functor is the composite
$$\Shi\left(\Affku,\et\right) \stackrel{\Shi\left(\left(\blank\right)_{\et}\right)}{\longlonglongrightarrow} \Topi \stackrel{\Shape}{\longlongrightarrow} \Pro\left(\cS\right),$$ and is denoted by $\Pi^{\et}_\i$. For $F$ an $\i$-sheaf on $\left(\Affku,\et\right),$ its \textbf{\'etale homotopy type} is $\Shape\left(\Shi\left(F_{\et}\right)\right),$ the shape of its small \'etale $\i$-topos.
\end{definition}

We also introduce a slight variant:

\begin{definition}
The \textbf{hyper-\'etale fundamental $\i$-groupoid} functor is the composite 
$$\Shi\left(\Affku,\et\right) \stackrel{\Shi\left(\left(\blank\right)_{\et}\right)}{\longlonglongrightarrow} \Topi \stackrel{\mathbb{H}yp}{\longlongrightarrow} \Topi \stackrel{\Shape}{\longlongrightarrow} \Pro\left(\cS\right),$$ where $\mathbb{H}yp$ is the hypercompletion functor. We denote this composite by $\Pi^{\mathbb{H}\mbox{-}\et}_\i.$ For $F$ an $\i$-sheaf on $\left(\Affku,\et\right),$ its \textbf{hyper-\'etale homotopy type} is $\Shape\left(\Hshi\left(F_{\et}\right)\right),$ the shape of the hypercompletion of its small \'etale $\i$-topos.
\end{definition}

\begin{remark}\label{rmk:hypercompletion functor}
The process of hypercompleting an $\i$-topos is indeed functorial. Let $\LPshi\left(\mathfrak{Top}_\i\right)$ denote the $\i$-category of large presheaves of $\i$-groupoids on the $\i$-category of $\i$-topoi, and consider the inclusion $$q:\mathfrak{Top}^{\mathbb{HC}}_\i \hookrightarrow \mathfrak{Top}_\i$$ of the full subcategory of hypercomplete $\i$-topoi. Then, by \cite[Proposition 6.5.2.13]{htt}, for any $\i$-topos $\cE,$ if $y\left(\cE\right)$ is its associated representable (large) presheaf, the presheaf $q^*y\left(\cE\right)$ on $\mathfrak{Top}^{\mathbb{HC}}_\i$ is representable by the hypercompletion $\widehat{\cE}.$ It follows that there is a canonical natural equivalence making the following diagram commute
$$\xymatrix{\mathfrak{Top}_\i \ar@{-->}[rrd]_-{\mathbb{H}yp} \ar@{^{(}->}[r]^-{y} & \LPshi\left(\mathfrak{Top}_\i\right) \ar[r]^-{q^*} & \LPshi\left(\mathfrak{Top}^{\mathbb{HC}}_\i\right)\\
& & \mathfrak{Top}^{\mathbb{HC}}_\i. \ar@{^{(}->}[u]^-{y}}$$ Moreover, this canonical equivalence component-wise $$y\left(\widehat{\cE}\right) \stackrel{\sim}{\longrightarrow} q^*y\left(\cE\right)$$ under the Yoneda lemma corresponds to the canonical geometric morphism $$\epsilon_{\cE}:\widehat{\cE} \hookrightarrow \cE,$$ and hence $\mathbb{H}yp$ is right adjoint to the canonical inclusion $$\mathfrak{Top}^{\mathbb{HC}} \hookrightarrow \mathfrak{Top}_\i,$$ with counit $$\epsilon: q \circ \mathbb{H}yp \Rightarrow id_{\mathfrak{Top}_\i}.$$
\end{remark}

\begin{remark}
The idea of using the small \'etale site of a scheme to define its \'etale homotopy type is not new, and goes back to Artin and Mazur in \cite{ArtinMazur}. The \'etale homotopy type of Artin and Mazur was further refined by Friedlander in \cite{Friedlander} and generalized to simplicial schemes. Finally, in \cite[Section 3.6]{dagxiii}, Lurie defines the \'etale homotopy type of a spectral Deligne-Mumford stack $\cX$ as the shape of its underlying $\i$-topos.
\end{remark}

\begin{remark}\label{rmk: closely related}
The hyper-\'etale homotopy type of a scheme $X$ is closely related to its \'etale homotopy type as defined in \cite[Definition 9.6 on p. 114]{ArtinMazur}. Lets illustrate this. For simplicity, lets assume that $X$ is locally Noetherian, so that its small \'etale site is locally connected by \cite[I 6.1.9]{EGA1}. Let $Z$ be a space in $\cS.$ Then, as a left exact functor $$\Pi^{\mathbb{H}\mbox{-}\et}_\i\left(X\right):\cS \to \cS,$$ we have $$\Pi^{\mathbb{H}\mbox{-}\et}_\i\left(X\right)\left(Z\right)=\Gamma_{\Hshi\left(X_{\et}\right)}\Delta_{\Hshi\left(X_{\et}\right)}\left(Z\right),$$ that is, it assigns $Z$ the space of sections of the hypersheafification on the constant presheaf with value $Z.$ Hypersheafification of a presheaf $F$ can be constructed in one-step (p. 672 of \cite{htt}) as follows:
$$F^\dagger\left(X\right) = \underset{U^\bullet \to X} \colim \left[\underset{n \in \Delta} \lim F\left(U^n\right)\right],$$ with the colimit ranging over a suitable filtered category of split hypercovers by connected objects in the small \'etale site for $X.$ For such a hypercover $$U^\bullet \to X$$ let $\Pi_0\left(U^n\right)$ be the set of connected components of $U^n$ in the sense of p. 111 of \cite{ArtinMazur}, and denote the corresponding simplicial set by $\pi\left(U^\bullet\right).$ By abuse of notation, we will denote the associated object in $\cS$ by the same symbol. Note that we have $$\pi\left(U^\bullet\right)= \underset{n \in \Delta^{op}} \colim \Pi_0\left(U^n\right).$$ We thus have that $$\Pi^{\mathbb{H}\mbox{-}\et}_\i\left(X\right)\left(Z\right)= \underset{U^\bullet \to X} \colim \left[\underset{n \in \Delta} \lim \prod_{a \in \Pi_0\left(U^n\right)} Z\right],$$ i.e.
\begin{eqnarray*}
\Hom_{\Pro\left(\cS\right)}\left(\Pi^{\mathbb{H}\mbox{-}\et}_\i\left(X\right),j\left(Z\right)\right)&=&  \underset{U^\bullet \to X} \colim \left[\underset{n \in \Delta} \lim \prod_{a \in \Pi_0\left(U^n\right)} Z\right]\\
&\simeq& \underset{U^\bullet \to X} \colim \Hom_{\cS}\left(\underset{n \in \Delta^{op}} \colim \Pi_0\left(U^n\right),Z\right)\\
&\simeq& \underset{U^\bullet \to X} \colim \Hom_{\cS}\left(\pi\left(U^\bullet\right),Z\right)\\
&\simeq& \Hom_{\Pro\left(\cS\right)}\left(\underset{U^\bullet \to X} \lim \pi\left(U^\bullet\right),j\left(Z\right)\right),
\end{eqnarray*}
so we conclude that $\Pi^{\mathbb{H}\mbox{-}\et}_\i\left(X\right)$ can be identified with the pro-space $\underset{U^\bullet \to X} \lim \pi\left(U^\bullet\right).$ Comparing this to the Verdier functor on p. 112 of \cite{ArtinMazur}, we see the only difference between $\Pi^{\mathbb{H}\mbox{-}\et}_\i\left(X\right)$ and the Artin-Mazur \'etale homotopy type of $X$ is that $\Pi^{\mathbb{H}\mbox{-}\et}_\i\left(X\right)$ is a pro-object in the $\i$-category of spaces, whereas the Artin-Mazur \'etale homotopy type is a pro-object in the \emph{homotopy} category of spaces.
\end{remark}

We also note a recent result of Hoyois. First we will introduce some notation. Let $$q:\Set^{\Delta^{op}} \to \cS$$ be the functor sending a simplicial set to its associated $\i$-groupoid. We can realize this concretely e.g. as $$\Set^{\Delta^{op}}=\Psh\left(\Delta\right) \hookrightarrow \Pshi\left(\Delta\right) \stackrel{\colim\left(\blank\right)}{\longlongrightarrow} \cS.$$ Notice that $q$ induces a well-defined functor $$\Pro\left(q\right):\Pro\left(\Set^{\Delta^{op}}\right) \to \Pro\left(\cS\right).$$

\begin{proposition}(\cite[Corollary 3.4]{Hoyois})\label{prop: Hoyois}\\
Let $X$ be a locally connected scheme (e.g. a locally Noetherian scheme). Denote by $\mathfrak{Fr}^{\et}\left(X\right)$  the \'etale homotopy type of $X$, as defined by Friedlander (\cite[Definition 4.4]{Friedlander}). Then $$\Pro\left(q\right)\left(\mathfrak{Fr}^{\et}\left(X\right)\right)\simeq \Shape\left(\Hshi\left(X_{\et}\right)\right),$$ i.e. the pro-space associated to $\mathfrak{Fr}^{\et}\left(X\right)$ agrees with $\Pi^{\mathbb{H}\mbox{-}\et}_\i\left(X\right).$
\end{proposition}


\begin{definition}
The \textbf{profinite \'etale fundamental $\i$-groupoid} functor is the composite
$$\Shi\left(\Affku,\et\right) \stackrel{\Shi\left(\left(\blank\right)_{\et}\right)}{\longlonglongrightarrow} \Topi \stackrel{\Shape^{\Prof}}{\longlongrightarrow} \Profs,$$ and is denoted by $\widehat{\Pi}^{\et}_\i$. For $F$ an $\i$-sheaf on $\left(\Affku,\et\right),$ its \textbf{profinite \'etale homotopy type} is $\Shape^{\Prof}\left(\Shi\left(F_{\et}\right)\right),$ the shape of its small \'etale $\i$-topos.
\end{definition}

\begin{remark}
There is no need to introduce a hyper-\'etale variant of the profinite \'etale $\i$-groupoid functor since by Proposition \ref{prop:hypersame}, for each \'etale $\i$-sheaf $F$ we have an induced equivalence of profinite spaces $$\Shape^{\Prof}\left(\Hshi\left(F_{\et}\right)\right) \stackrel{\sim}{\longrightarrow} \Shape^{\Prof}\left(\Shi\left(F_{\et}\right)\right).$$ In particular, this has an effect of ``erasing the difference'' between the \'etale homotopy type and the hyper-\'etale homotopy type of a stack.
\end{remark}

\subsection{A concrete description of the \'etale homotopy type}

Thus far, we have succeeded in generalizing the previously existing definitions of \'etale homotopy type to a definition that makes sense for arbitrary higher stacks on the \'etale site. However, when the stack in question is not Deligne-Mumford, this construction is a bit opaque, since it involves taking the shape of a colimit of $\i$-topoi indexed by the right fibration associated to the stack in question. In this subsection, we will show that nonetheless, the pro-space associated to such a higher stack has a natural concrete description.

Consider the essentially unique geometric morphism $$\Shi\left(\Affku,\et\right) \to \cS$$ to the terminal $\i$-topos. This is represented by an adjoint pair $\Delta^{\et} \dashv \Gamma_{\et},$ and $\Delta^{\et}$ is left exact. Moreover, $\Delta^{\et}$ assigns a space $Z$ the sheafification of the constant presheaf with value $Z,$ and $\Gamma_{\et}$ assigns an $\i$-sheaf $F$ the value $F\left(\Spec\left(k\right)\right).$

Let $F$ be an object of $\Shi\left(\Affku,\et\right).$ Consider the composition of functors
$$\cS \stackrel{\Delta^{\et}}{\longlongrightarrow} \Shi\left(\Affku,\et\right)
 \stackrel{y}{\longhookrightarrow} \Pshi\left(\Shi\left(\Affku,\et\right)\right) \stackrel{ev_F}{\longlongrightarrow} \cS,$$
 where $$ev_F: \Pshi\left(\Shi\left(\Affku,\et\right)\right) \to \cS$$ is the functor evaluating a presheaf $G$ on $\Shi\left(\Affku,\et\right)$ at the object $F.$ Since limits in a functor $\i$-category are computed object-wise, and since $\Delta^{\et}$ is left exact, the above composition is also left-exact, hence a pro-space. Lets denote this pro-space by $l\left(F\right).$ The pro-space $l\left(F\right)$ is given by the simple formula 
\begin{equation}\label{eq:el}
l\left(F\right)\left(Z\right)=\Hom_{\Shi\left(\Affku,\et\right)}\left(F,\Delta^{\et}\left(Z\right)\right).
\end{equation}
This can be formulated more abstractly as follows. Let $\cE$ be an arbitrary $\i$-topos and let $$\Delta^{\cE} \dashv \Gamma_{\cE}$$ denote the essentially unique geometric morphism $e:\cE \to \cS.$ Let $E$ be an object in $\cE$ and let $$\pi_E:\cE/E \to \cE$$ denote the associated \'etale geometric morphism. Then the composite $$\cE/E \stackrel{\pi_E}{\longlongrightarrow} \cE \stackrel{e}{\longrightarrow} \cS$$ is the essentially unique geometric morphism from $\cE/E$ to $\cS,$ so $$\Delta^{\cE/E}\simeq\pi_{E}^*  \circ \Delta^{\cE}.$$ It follows that for $Z$ in $\cS$ we have:
 \begin{eqnarray*}
 \Shape\left(\cE/E\right)\left(Z\right)&\simeq&\Gamma_{\cE/E}\left(\Delta^{\cE/E}\left(Z\right)\right)\\
 &\simeq& \Gamma\left(E \times \Delta^{\cE}\left(Z\right) \to E\right)\\
 &\simeq& \Hom_{\cE}\left(E,\Delta^{\cE}\left(Z\right)\right).
 \end{eqnarray*}
 By \cite[Proposition 6.3.5.14]{htt}, the assignment $E \mapsto \cE/E$ assembles into a colimit preserving functor $$\chi:\cE \to \Topi.$$ By composition, we get a colimit preserving functor $$\cE \stackrel{\chi}{\longrightarrow} \Topi \stackrel{\Shape}{\longlongrightarrow} \Pro\left(\cS\right)$$ sending an object $E$ of $\cE$ to $\Shape\left(\cE/E\right).$ By the above discussion, we see that for $F$ an object of the $\i$-topos $\Shi\left(\Affku,\et\right),$ the pro-space $l\left(F\right)$ is nothing but $\Shape\left(\Shi\left(\Affku,\et\right)/F\right),$ and hence the assignment $F \mapsto l\left(F\right)$ assembles into a colimit preserving functor $$\Shi\left(\Affku,\et\right) \stackrel{\chi}{\longrightarrow} \Topi \stackrel{\Shape}{\longlongrightarrow} \Pro\left(\cS\right).$$

\begin{lemma}
Consider the ringed $\i$-topos $\Spec_{\et}\left(k\right)$ and denote by $$\delta:\cS \to \Shi\left(k_{\et}\right)$$ the inverse image functor of the unique geometric morphism $$\Shi\left(k_{\et}\right) \to \cS$$ from the underlying $\infty$-topos of $\Spec_{\et}\left(k\right)$ to the terminal $\i$-topos of spaces. Let $Z$ be an arbitrary space. Then $\Spec_{\et}\left(k\right)/\delta\left(Z\right)$ is a Deligne-Mumford scheme, whose functor of points is the stack $\Delta^{\et}\left(Z\right).$
\end{lemma}

\begin{proof}
It follows from \cite[Proposition 2.3.10]{dagv} that $\Spec_{\et}\left(k\right)/\delta\left(Z\right)$ is a Deligne-Mumford scheme. It therefore suffices to show that its functor of points is $\Delta^{\et}\left(Z\right)$. Consider the composition of functors
$$\cS \stackrel{\delta}{\longrightarrow} \Shi\left(k_{\et}\right) \stackrel{\sim}{\longrightarrow} \mathfrak{DM}\left(k\right)_\i^{\cU,\et}/\Spec_{\et}\left(k\right) \to \mathfrak{DM}\left(k\right)_\i^{\cU,\et} \to \mathfrak{DM}\left(k\right)_\i^{\cU} \hookrightarrow \Shi\left(\Affku,\et\right).$$
By \cite[Proposition 2.3.5]{dagv} and \cite[Proposition 5.2.11]{higherdave}, the composition preserves small colimits. Moreover, it sends the one point space $*$ to $\Spec\left(k\right),$ which is terminal. The functor $\Delta^{\et}$ also has this property, and since the above composite and $\Delta^{\et}$ are both colimit preserving functors out of $\cS=\Pshi\left(*\right),$ they must agree, but this is exactly what we wanted to show, since the composite sends a space $Z$ to the functor of points of $\Spec_{\et}\left(k\right)/\delta\left(Z\right).$
\end{proof}

 \begin{theorem}\label{thm: concrete}
 There is a canonical equivalence of functors $$\Pi^{\et}_\i \stackrel{\sim}{\longrightarrow} \Shape \circ \chi .$$ In particular, for any $\i$-sheaf $F$ on $\left(\Affku,\et\right),$ there is a canonical equivalence of pro-spaces $$\Pi^{\et}_\i\left(F\right) \stackrel{\sim}{\longrightarrow} l\left(F\right),$$ where $l\left(F\right)$ is defined as in equation (\ref{eq:el}).
 \end{theorem}

\begin{proof}
Since both functors $\Shape \circ \chi$ and $\Pi^{\et}_\i$ are colimit preserving functors $$\Shi\left(\Affku,\et\right) \to \Pro\left(\cS\right),$$ by \cite[Proposition 5.5.4.20 and Theorem 5.1.5.6]{htt}, it suffices to show that both functors agree up to equivalence when restricted to affine schemes. Note that both $\Spec\left(A\right)$ and $\Delta^{\et}\left(Z\right)$ the are functor of points of Deligne-Mumford schemes, and the functor of points construction is a fully faithful embedding of Deligne-Mumford schemes into $\Shi\left(\Affku,\et\right).$ It follows that the canonical map $$\Hom_{\Dmsch}\left(\Spec_{\et}\left(A\right),\Spec_{\et}\left(k\right)/\delta\left(Z\right)\right) \to \Hom_{\Shi\left(\Affku,\et\right)}\left(\Spec\left(A\right),\Delta^{\et}\left(Z\right)\right)\simeq l\left(\Spec\left(A\right)\right)\left(Z\right)$$ is an equivalence of $\i$-groupoids.

By \cite[Remark 2.3.20]{dagv}, the following is a pullback diagram in $\mathfrak{DM}\left(k\right)_\i^{\cU}$:
$$\xymatrix{\Spec_{\et}\left(A\right)/f^*\delta\left(Z\right) \ar[r] \ar[d] & \Spec_{\et}\left(k\right)/\delta(Z)\ar[d]\\
\Spec_{\et}\left(A\right) \ar[r]^-{f} & \Spec_{\et}\left(k\right),}$$ where $f$ is the map whose functor of points is the unique map to $\Spec\left(k\right).$
Since $\Spec_{\et}\left(k\right)$ is the terminal Deligne-Mumford scheme, it follows that $$\Hom_{\Dmsch}\left(\Spec_{\et}\left(A\right),\Spec_{\et}\left(k\right)/\delta\left(Z\right)\right)$$ is equivalent to the space of sections of the \'etale map  $$\Spec_{\et}\left(A\right)/f^*\delta\left(Z\right) \to \Spec_{\et}\left(A\right),$$ and since any section of an \'etale map is \'etale, this is in turn the space of maps in the slice category $$\Hom_{\mathfrak{DM}\left(k\right)_\i^{\cU,\et}/\Spec_{\et}\left(A\right)}\left(id_{\Spec_{\et}\left(A\right)},\Spec_{\et}\left(A\right)/f^*\delta\left(Z\right) \to \Spec_{\et}\left(A\right)\right).$$ Now, by \cite[Proposition 2.3.5]{dagv}, this is equivalent to the space of maps $$\Hom_{\Shi\left(A_{\et}\right)}\left(1,f^*\delta\left(Z\right)\right)=\Gamma^A\left(f^*\delta\left(Z\right)\right).$$ Since $\cS$ is the terminal $\i$-topos, it follows that $f^*\delta\simeq\Delta^{\et}_A,$ (where $\Delta^{\et}_A$ is the inverse image of the unique geometric morphism to $\cS$) and hence 
\begin{eqnarray*}
\Gamma^A\left(f^*\delta\left(Z\right)\right)&\simeq& \Gamma^A\left(\Delta^{\et}_A\left(Z\right)\right)\\
&=&\Shape\left(\Spec\left(A\right)\right)\left(Z\right)\\
&=& \Pi^{\et}_\i\left(\Spec\left(A\right)\right)\left(Z\right).
\end{eqnarray*}

\end{proof}

\section{Cohomology with coefficients in a local system}\label{sec: cohomology}
In this section we give a careful introduction to the concept of cohomology with coefficients in a local system of abelian groups using the modern language of $\i$-categories. We work this out first for the case of spaces, and then for an arbitrary $\i$-topos, and link these definitions with the classical definition of cohomology with twisted coefficients in a topos. The material in this section will play a pivotal role in proving the main theorem of this paper.

\subsection{Topological case}
In this subsection, we will explain how to define the cohomology of a space $X$ with coefficients in a local system of abelian groups by using classifying spaces. The basic idea is not new and goes back to \cite{robinson,coh}, and we benefited greatly from discussion with Achim Krause. In what follows, we formulate cohomology of local systems on spaces in the natural setting of the $\i$-category $\cS$ of spaces.

\subsubsection{Preliminaries on $\i$-groupoids}

Let $X$ be a space in $\cS.$ Regarding $X$ as an $\i$-groupoid (and hence as an $\i$-category), by the proof of \cite[Corollary 5.3.5.4]{htt}, there is a canonical equivalence of $\i$-categories $$\cS/X \simeq \Pshi\left(X\right).$$ Let us unravel this equivalence a little. First note that since $X$ is an $\i$-\emph{groupoid}, $X$ is naturally equivalent to its opposite $\i$-category $X^{op},$ so we have a canonical equivalence $\Pshi\left(X\right)\simeq \Fun\left(X,\cS\right).$ It will be convenient to phrase things in terms of $\Fun\left(X,\cS\right)$ instead:\\

Given an object $f:Y \to X$ of $\cS/X,$ it defines a functor $F:X \to \cS=\iGpd,$ by assigning to each object $x$ of the $\i$-groupoid $X$, the $\i$-groupoid of lifts
$$\xymatrix{& Y \ar[d]^-{f}\\
\ast \ar@{-->}[ur] \ar[r]^-{x} & X.}$$ By abstract nonsense, this is the same as the $\i$-groupoid of lifts in the pullback diagram
$$\xymatrix{\ast \times^{x,f}_{X} Y \ar[r] \ar[d]& Y \ar[d]^-{f}\\
\ast \ar@/^2pc/@{-->}[u] \ar[r]^-{x} & X,}$$
in other words the space $$\Hom\left(*,\ast \times^{x,f}_{X} Y\right) \simeq \ast \times^{x,f}_{X} Y.$$ In particular, if $f:X \to Y$ arises from a continuous map $\tilde f$ of topological spaces, $F\left(x\right)$ is the homotopy fiber of $\tilde f$ over $x.$ Conversely, given an arbitrary functor $F:X \to \iGpd,$ it corresponds to a left fibration $$\pi_F:\int_X F \to X,$$ which since $X$ is a Kan complex (when regarded as a quasicategory) is also a Kan fibration, and hence $\pi_F$ is a map of Kan complexes, corresponding to an object of $\cS/X.$ It follows from the proof of \cite[Corollary 5.3.5.4]{htt} that as an $\i$-category, using the identification $\Fun\left(X,\cS\right)\simeq \Pshi\left(X\right),$ $\int_X F$ is the full subcategory of the slice category $\Pshi\left(X\right)/F$ on those morphisms whose domain is a representable presheaf, and this $\i$-category is an $\i$-groupoid.

\begin{proposition}\label{prop:colim}
If $F:X \to \cS$ is a functor, its associated left fibration $$\pi_F:\int_X F \to X$$ can canonically be identified with the canonical map $$\colim F \to X.$$
\end{proposition}

\begin{proof}
Consider the composite $$\Pshi\left(X\right) \stackrel{\sim}{\longrightarrow} \cS/X \to \cS,$$ where the functor $\cS/X \to \cS$ is the forgetful functor, which is colimit preserving. Let $x$ be an object of the $\i$-groupoid $X.$ Then the representable presheaf $y\left(x\right)$ gets mapped as
$$y\left(x\right) \mapsto x:* \to X \mapsto *$$ under the above composite. It follows from \cite[Theorem 5.1.5.6]{htt} that the composite must in fact be equivalent to the left Kan extension of the terminal functor $X \to \cS$ (sending everything to the contractible space) along the Yoneda embedding $$y:X \hookrightarrow \Pshi\left(X\right).$$ It follows that the composite is the functor assigning a presheaf $F$ its colimit. The result now follows from \cite[Proposition 1.2.13.8]{htt}.
\end{proof}

\subsubsection{Universal fibrations}

Let $\mbox{Cat}_\i$ be the $\i$-category of small $\i$-categories. Recall that for $\sC$ an $\i$-category, given a functor $$F:\sC \to \mbox{Cat}_\i$$ there is an associated coCartesian fibration
$$\underset{\sC} \int F \to \sC.$$ This is an $\i$-categorical analogue of the classical Grothendieck construction from category theory, which associates a functor $$F:\sD \to \mbox{Cat},$$ with $\sD$ a small category, to its associated cofibered category $$\underset{\sD} \int F \to \sD.$$ This construction is part of an equivalence of $\i$-categories between the functor category $\Fun\left(\sC,\mbox{Cat}_\i\right)$ and the $\i$-category of coCartesian fibrations over $\sC$ (more precisely the $\i$-category associated to the covariant model structure on marked simplicial sets over $\sC$).

\begin{definition}
Let $id:\mbox{Cat}_\i \to \mbox{Cat}_\i$ be the identity functor. The \textbf{universal coCartesian fibration} is its associated coCartesian fibration
$$\cZ:=\underset{\mbox{Cat}_\i} \int{\!\!\!id}  \to \mbox{Cat}_\i.$$
\end{definition}

The universal property of the above coCartesian fibration, which justifies its name, is that if $F:\sC \to \mbox{Cat}_\i$ is a functor, then the following is a pullback diagram
$$\xymatrix{\underset{\sC} \int F \ar[d] \ar[r] & \cZ \ar[d]\\
\sC \ar[r]^-{F} & \mbox{Cat}_\i}$$
(in the $\i$-category $\widehat{\mbox{Cat}}_\i$ of large $\i$-categories). See \cite[Section 3.3.2]{htt}.

Consider now the canonical inclusion $$\cS=\iGpd \hookrightarrow \mbox{Cat}_\i$$ of the $\i$-category of spaces (i.e. $\i$-groupoids) into the $\i$-category of small $\i$-categories. Denote this inclusion by $I.$ Then the associated coCartesian fibration
$$\underset{\cS} \int I \to \cS$$ is in fact a left fibration, since $I$ factors through the identity functor $$id:\cS \to \cS,$$ and can in fact be identified with the left fibration associated to the identity functor of $\cS.$ Also, we can identify it as the fibered product $$\cS \times_{\mbox{Cat}_{\i}} \cZ \to \cS.$$

\begin{definition}
The left fibration $\cZ^{\cS}:=\cS \times_{\mbox{Cat}_{\i}} \cZ \to \cS$ is the \textbf{universal left fibration}.
\end{definition}

It follows immediately by pasting pullback diagrams that if $F:\sC \to \cS$ is any functor, then its associated left fibration fits in a pullback square
$$\xymatrix{\underset{\sC} \int F \ar[d] \ar[r] & \cZ^\cS \ar[d]\\
\sC \ar[r]^-{F} & \cS.}$$

The following lemma follows immediately:

\begin{lemma}\label{lem:pbfib}
Let $f:\sC \to \sD$ be a functor between $\i$-categories and let $G:\sD \to \cS$ be another functor. Then the following diagram is a pullback diagram:
$$\xymatrix{ \int_\sC \left(G \circ f\right) \ar[d] \ar[r] & \int_\sD G \ar[d]\\
\sC \ar[r]^-{f} & \sD.}$$
\end{lemma}

We will now identify the universal left fibration as a very familiar functor:

\begin{proposition}
The universal left fibration can be canonically identified with the functor $$\cS_{\ast} \to \cS$$ from pointed spaces to spaces which forgets the base point.
\end{proposition}

\begin{proof}
The opposite functor $$\left(\cZ^\cS\right)^{op} \to \cS^{op}$$ is a right fibration corresponding to the functor $$\left(\cS^{op}\right)^{op}=\cS \stackrel{id}{\longrightarrow} \cS.$$ It suffices to show that the right fibration $$\underset{\cS^{op}} \int id \to \cS^{op}$$ can be canonically identified with the opposite of the functor $$\cS_{\ast} \to \cS.$$ Let $\widehat{\cS}$ denote the $\i$-category of large spaces, i.e. the $\i$-category of spaces in the Grothendieck universe $\cV$ of large sets, and denote by $\cJ$ the composite $$\left(\cS^{op}\right)^{op}=\cS \stackrel{id}{\longrightarrow} \cS \hookrightarrow \widehat{\cS}.$$ After changing universes, $\cS$ may be considered as a small $\i$-category and $\cJ$ as a presheaf on $\cS^{op}.$ By the proof of \cite[Corollary 5.3.5.4]{htt}, the underlying $\i$-category of the right fibration associated to $\cJ$ can be identified with the full subcategory of the slice category of large presheaves $$\LPshi\left(\cS^{op}\right)/\cJ$$ on those objects of the form $$\widehat{x}:y\left(X\right) \to \cJ,$$ where $y\left(X\right)$ is a representable presheaf on $\cS^{op},$ where the fibration sends such an object to $X.$ By the Yoneda lemma, since $$\Hom_{\LPshi\left(\cS^{op}\right)}\left(y\left(X\right),\cJ\right)\simeq \cJ\left(X\right)=X,$$ these objects can be identified with pointed spaces, where a point $$x:* \to X$$ can be identified with a map $$\widehat{x}:y\left(X\right) \to \cJ.$$ By the description of mapping spaces in a slice $\i$-category \cite[Proposition 5.5.5.12]{htt}, if $\left(X,x\right)$ and $\left(Y,y\right)$ are pointed spaces, there is a pullback diagram
$$\xymatrix{\Hom_{\LPshi\left(\cS^{op}\right)/\cJ}\left(\left(X,x\right),\left(Y,y\right)\right) \ar[r] \ar[d] & \Hom_{\LPshi\left(\cS^{op}\right)}\left(y\left(X\right),y\left(Y\right)\right) \ar[d]\\
\ast \ar[r]^-{\widehat{x}} & \Hom_{\LPshi\left(\cS^{op}\right)}\left(y\left(X\right),\cJ\right)}$$
where the right most vertical arrow is induced by composition with the morphism $$y\left(Y\right) \to \cJ$$ corresponding to the base point $y$. By the Yoneda lemma, this diagram can be identified with the following diagram
$$\xymatrix{\Hom_{\LPshi\left(\cS^{op}\right)/\cJ}\left(\left(X,x\right),\left(Y,y\right)\right) \ar[r] \ar[d] & \Hom_{\cS}\left(Y,X\right) \ar^-{ev_y}[d]\\
\ast \ar[r]^-{x} & X,}$$
where $ev_y$ is the map induced by evaluation at the base point $y.$ It follows that the mapping space $\Hom_{\LPshi\left(\cS^{op}\right)/\cJ}\left(\left(X,x\right),\left(Y,y\right)\right)$ can be canonically identified with the mapping space in pointed spaces $\Hom_{\cS_*}\left(\left(Y,y\right),\left(X,x\right)\right).$ The result now follows.
\end{proof}

\subsubsection{Classifying spaces and cohomology with coefficients local systems}
Fix $A$ an abelian group and let $n>0$ be an integer. Denote by $\cS_{K\left(A,n\right)}$ the full subcategory of spaces on the single space $K\left(A,n\right),$ and denote by $B\Aut\left(K\left(A,n\right)\right)$ its maximal sub-Kan-complex, i.e. the $\i$-groupoid obtained by throwing out all the arrows which are not equivalences. Define 
the space $\Aut\left(K\left(A,n\right)\right)$ to be the mapping space $\Hom_{B\Aut\left(K\left(A,n\right)\right)}\left(K\left(A,n\right),K\left(A,n\right)\right).$ Concretely, $\Aut\left(K\left(A,n\right)\right)$ is the space of self homotopy equivalences of $K\left(A,n\right).$ Denote by $\Theta_n$ the canonical inclusion
$$B\Aut\left(K\left(A,n\right)\right) \to \cS_{K\left(A,n\right)} \hookrightarrow \cS.$$ 
Denote by $\theta_n:U_n^A \to B\Aut\left(K\left(A,n\right)\right)$ the left fibration classified by the functor $\Theta_n,$ which is simply a map of spaces.

Denote by $\Aut_*\left(K\left(A,n\right)\right)$ the space of homotopy equivalences of $K\left(A,n\right)$ that preserve the base-point. It is the subcategory of the $\i$-groupoid $$\mathbf{End}_*\left(K\left(A,n\right)\right)=\Hom_{\cS_*}\left(K\left(A,n\right),K\left(A,n\right)\right)$$ of pointed endomorphisms of $K\left(A,n\right)$ on those which are equivalences. By \cite[Theorem 5.1.3.6]{higheralgebra} the $n$-fold loop space functor $$\Omega^n:\cS_{*}^{\ge n} \to \mbox{Mon}^{gp}_{\mathbb{E}_n}\left(\cS\right)$$ from the $\i$-category of pointed $n$-connective spaces to the $\i$-category of grouplike $\mathbb{E}_n$-spaces is an equivalence. Since $K\left(A,n\right)$ is $n$-connective, it follows that 
\begin{eqnarray*}
\mathbf{End}_*\left(K\left(A,n\right)\right)&=& \Hom_{\cS_{*}^{\ge n}}\left(K\left(A,n\right),K\left(A,n\right)\right)\\
&\simeq& \Hom_{\mbox{Mon}^{gp}_{\mathbb{E}_n}}\left(\Omega^n K\left(A,n\right),\Omega^n K\left(A,n\right)\right)\\
&\simeq& \Hom_{\mbox{Mon}^{gp}_{\mathbb{E}_n}}\left(A,A\right).
\end{eqnarray*}
Since $A$ is a discrete group, we have finally that $$\mathbf{End}_*\left(K\left(A,n\right)\right)\simeq \Hom_{\textbf{Grp}}\left(A,A\right).$$ It follows that $$\Aut_*\left(K\left(A,n\right)\right)\simeq \Aut\left(A\right).$$ 
Note also that we have a pullback diagram in $\cS$:
$$\xymatrix{\Aut\left(A\right) \simeq \Aut_\ast\left(K\left(A,n\right)\right) \ar[r] \ar[d] & \Aut\left(K\left(A,n\right)\right) \ar[d]^-{ev_\ast}\\
\ast \ar[r] & K\left(A,n\right).}$$
Unwinding the definitions, since the inverse functor to $\Omega^n$ is $B^n,$ we have that the map above $$\Aut\left(A\right) \to \Aut\left(K\left(A,n\right)\right)$$ sends an automorphism $\varphi:A \stackrel{\sim}{\longrightarrow} A$ to the automorphism $$K\left(\varphi,n\right):K\left(A,n\right) \stackrel{\sim}{\longrightarrow} K\left(A,n\right).$$ Also, via the long exact sequence in homotopy groups from the fibration sequence associated to the above diagram, we conclude that 
$$\pi_0\left(\Aut\left(K\left(A,n\right)\right)\right)\cong \Aut\left(A\right)$$ and the only other non-trivial homotopy group is $$\pi_n\left(\Aut\left(K\left(A,n\right)\right)\right)\cong A.$$ In other words, we have an equivalence of spaces
$$\Aut\left(K\left(A,n\right)\right) \simeq \Aut\left(A\right) \times K\left(A,n\right).$$
In fact, we even have a semi-direct product decomposition $$\Aut\left(K\left(A,n\right)\right) \simeq \Aut\left(A\right)\ltimes K\left(A,n\right).$$

Recall that $\theta_n:U_n^A \to B\Aut\left(K\left(A,n\right)\right)$ is the left fibration classified by the canonical functor $$\Theta_n:B\Aut\left(K\left(A,n\right)\right) \to \cS.$$ Let us compute what $U_n^A$ is explicitly. As an $\i$-category, this is the full subcategory of the slice category $\Pshi\left(BAut\left(K\left(A,n\right)\right)\right)/\Theta_n$ on those maps $G \to \Theta_n$ with $G$ a representable presheaf. But, $B \Aut\left(K\left(A,n\right)\right)$ only has one object, call it $\star$. Denote by $y\left(\star\right)$ its associated presheaf, which sends $\star$ to $\Aut\left(K\left(A,n\right)\right)$. By the Yoneda lemma, we have $$Hom\left(y\left(\star\right),\Theta_n\right)\simeq K\left(A,n\right).$$ It follows that the $\i$-category $U_n^A$ also has a single object, call it  $V$. We can write $$V:y\left(\star\right) \to \Theta_n.$$ The space of maps $\Hom_{U_n^A}\left(V,V\right)$ is the space of maps in the slice category $$\Psh_\infty\left(B\Aut\left(K\left(A,n\right)\right)\right)/\Theta_n.$$ By \cite[Proposition 5.5.5.12]{htt}, we can identify $\Hom_{U_n^A}\left(V,V\right)$ with the fiber of the map
$$\Hom\left(y\left(\star\right),y\left(\star\right)\right) \to \Hom\left(y\left(\star\right),\Theta_n\right)$$ induced by composition with $V.$ By the Yoneda lemma, this is equivalent to the fiber of the map 
$$\Aut\left(K\left(A,n\right)\right) \to K\left(A,n\right).$$ Unwinding the definitions, we see that the above map $$\Aut\left(K\left(A,n\right)\right)\simeq \Aut\left(A\right) \times K\left(A,n\right) \to K\left(A,n\right)$$ is just the first projection. It follows that
\begin{itemize}
\item[1)] $\Hom_{U_n^A}\left(V,V\right) \simeq \Aut\left(A\right),$
\item[2)] The canonical map $\Hom_{U_n^A}\left(V,V\right) \to \Aut\left(K\left(A,n\right)\right)$ induced by the left fibration $$\theta_n:U_n^A \to B \Aut\left(K\left(A,n\right)\right)$$ sends an automorphism $$\varphi:A \stackrel{\sim}{\longrightarrow} A$$ to the automorphism $$K\left(\varphi,n\right):K\left(A,n\right) \stackrel{\sim}{\longrightarrow} K\left(A,n\right).$$
\end{itemize}
From $1)$ we conclude that $U_n^A=K\left(\Aut\left(A\right),1\right).$ Notice that $K\left(\Aut\left(A\right),1\right)$ is a $1$-type, hence a groupoid. Viewing it as a groupoid, it is the groupoid with one object $\star$ such that $$\Hom\left(\star,\star\right)=\Aut\left(A\right).$$ As such, there is a canonical functor into the category of abelian groups $$\chi_A:K\left(\Aut\left(A\right),1\right) \to Ab,$$ sending $\star$ to $A$ and each automorphism of $A$ to itself. Composition with the $n^{th}$ Eilenberg-MacLane functor then yields a functor into spaces 
$$K\left(\Aut\left(A\right),1\right) \stackrel{\chi_A}{\longlongrightarrow} Ab \stackrel{ K\left(\blank,n\right)}{\longlongrightarrow} \cS.$$ This functor sends the single object $\star$ to $K\left(A,n\right)$ and sends each automorphism $$\varphi:A \stackrel{\sim}{\longrightarrow}A$$ to $$K\left(\varphi,n\right):K\left(A,n \right) \stackrel{\sim}{\longrightarrow} K\left(A,n \right),$$ hence by $2),$ there is a factorization
$$\xymatrix{& K\left(\Aut\left(A\right),1\right) \ar[d]^-{\chi_A} \ar@{-->}[ldd]_-{\theta_n}\\
& Ab \ar[d]^-{K\left(\blank,n\right)}\\
B\Aut\left(K\left(A,n\right)\right) \ar[r] & \cS.}$$

\begin{definition}
The map $\theta_n:K\left(\Aut\left(A\right),1\right) \to B\Aut\left(K\left(A,n\right)\right)$ is the \textbf{universal $K\left(A,n\right)$-fibration}.
\end{definition}

The following proposition justifies this terminology:

\begin{proposition}\label{prop: univ KAN fibration}
Let $g:Y \to X$ be any map of spaces whose fibers are all equivalent to $K\left(A,n\right).$ Then there is a pullback diagram
$$\xymatrix@C=2cm{Y \ar[d]_-{g} \ar[r] & K\left(\Aut\left(A\right),1\right) \ar[d]^-{\theta_n}\\
X \ar[r]_-{c_{g}} & B\Aut\left(K\left(A,n\right)\right).}$$
\end{proposition}

\begin{proof}
Under the equivalence $\cS/X \simeq \Fun\left(X,\cS\right),$ $g:Y \to X$ corresponds to a functor $$G:X \to \cS$$ that factors as
$$X \stackrel{c_g}{\longrightarrow} B\Aut\left(K\left(A,n\right)\right) \stackrel{\Theta_n}{\longlongrightarrow} \cS.$$
As $$\underset{B\Aut\left(K\left(A,n\right)\right)}\int \!\!\!\!\!\!\!\!\!\! \Theta_n \longrightarrow B\Aut\left(K\left(A,n\right)\right)$$ can be canonically identified with $$\theta_n:K\left(\Aut\left(A\right),1\right)  \to B\Aut\left(K\left(A,n\right)\right),$$ the result now follows for Lemma \ref{lem:pbfib}.
\end{proof}

A \textbf{local system} on a space $X$ with coefficients in an abelian group $A$ is usually defined in the connected case as a group homomorphism $$\pi_1\left(X\right) \to \Aut\left(A\right),$$ or in the non-connected case, as an action of the fundamental groupoid $\Pi_1\left(X\right)$ on $A,$ or equivalently, a functor of groupoids $$\Pi_1\left(X\right) \to K\left(\Aut\left(A\right),1\right).$$ This is the same data as a map $$\tau:X \to K\left(\Aut\left(A\right),1\right).$$

\begin{proposition}\label{prop:twistedcoh}
Let $n>0$ be an integer. Given a local system $\tau$ as above, the $n^{th}$-cohomology group of $X$ with coefficients in $\tau$ is in natural bijection with the set of homotopy classes of lifts

$$\xymatrix@R=3cm{ & & K\left(\Aut\left(A\right),1\right) \ar[d]^-{\theta_n}\\
X \ar[r]_-{\tau} \ar@{-->}[rru] & K\left(\Aut\left(A\right),1\right) \ar[r]_-{\theta_n} & B\Aut\left(K\left(A,n\right)\right).}$$
\end{proposition}

\begin{proof}
Let $\tau:X \to K\left(\Aut\left(A\right),1\right)$ be a local system with coefficients in $A.$ Notice that the composite 
\begin{equation}\label{eq:comp}
X \stackrel{\tau}{\longrightarrow} K\left(\Aut\left(A\right),1\right) \stackrel{
\chi_A}{\longlongrightarrow} Ab \stackrel{K\left(\blank,n\right)}{\longlongrightarrow} \cS$$ has a factorization of the form $$X \to \Pi_1\left(X\right) \stackrel{\tau'}{\longrightarrow} Ab  \stackrel{K\left(\blank,n\right)}{\longlongrightarrow} \cS.
\end{equation}
Denote by $L_X\left(\tau,n\right)$ the colimit of $$K\left(\tau',n\right):\Pi_1\left(X\right)\to \cS.$$ There is a canonical map $L_X\left(\tau,n\right) \to \Pi_1\left(X\right)$ and by \cite[Corollary 4.6]{coh}, there is a natural bijection between the set of homotopy classes of lifts
$$\xymatrix@C=2.5cm{& L_X\left(\tau,n\right) \ar[d]\\
X \ar[r] \ar@{-->}[ru]& \Pi_1\left(X\right)}$$
and the $n^{th}$ cohomology group of $X$ with coefficients in $\tau.$
Note that the space of such lifts is canonically homotopy equivalent to the space of sections
$$\xymatrix@C=2.5cm@R=1.5cm{X\times_{\Pi_1\left(X\right)} L_X\left(\tau,n\right) \ar[d] \ar[r] & L_X\left(\tau,n\right) \ar[d]\\
X \ar[r]  \ar@/^26pt/@{-->}[u]& \Pi_1\left(X\right).}$$
By Lemma \ref{lem:pbfib}, it follows that $$X\times_{\Pi_1\left(X\right)} L_X\left(\tau,n\right) \to X$$ can canonically be identified with the colimit of the composite (\ref{eq:comp}). Recall that the composite (\ref{eq:comp}) factors as 
$$X \stackrel{\tau}{\longrightarrow} K\left(\Aut\left(A\right),1\right) \stackrel{\theta_n}{\longlongrightarrow} B\Aut\left(K\left(A,n\right)\right) \stackrel{\Theta_n}{\longlongrightarrow} \cS.$$ Again by Lemma \ref{lem:pbfib}, it follows that the colimit of (\ref{eq:comp}) agrees with the left vertical arrow in the following pullback diagram
$$\xymatrix{X \times_{B\Aut\left(K\left(A,n\right)\right)} K\left(\Aut\left(A\right),1\right) \ar[d] \ar[r] & K\left(\Aut\left(A\right),1\right) \ar[d]^-{\theta_n}\\
X \ar[r]_-{\theta_n \circ \tau} & B\Aut\left(K\left(A,n\right)\right).}$$
Finally, since this is a pullback diagram, the space of sections of this map is homotopy equivalent to the space of lifts as in the statement of the proposition.
\end{proof}

Denote by $\underline{A}$ the underlying set of the abelian group $A,$ and denote by $$\Aut\left(A\right) \ltimes \underline{A},$$ the action groupoid associated to the action of $\Aut\left(A\right)$ on $\underline{A},$ i.e. the groupoid whose set of objects is $\underline{A}$ and whose set of arrows is $\Aut\left(A\right) \times \underline{A},$ where a pair $\left(\varphi,a\right)$ is an arrow from $a$ to $\varphi\left(a\right).$ Denote by $$\theta_0:\Aut\left(A\right) \ltimes \underline{A} \to K\left(\Aut\left(A\right),1\right)$$ the functor sending $\underline{A}$ to the unique object $\star$ and sending a pair $\left(\varphi,a\right)$ to $\varphi.$

\begin{proposition}\label{prop:0th coh}
Given a space $X$ and a local system $\tau:X \to K\left(\Aut\left(A\right),1\right),$ the $0^{th}$-cohomology group of $X$ with coefficients in $\tau$ is in natural bijection with the set of homotopy classes of lifts

$$\xymatrix{ & \Aut\left(A\right) \ltimes \underline{A} \ar[d]^-{\theta_0}\\
X \ar[r]_-{\tau} \ar@{-->}[ru] & K\left(\Aut\left(A\right),1\right).}$$
\end{proposition}

\begin{proof}
Consider the composite
$$K\left(\Aut\left(A\right),1\right) \stackrel{\chi_A}{\longlongrightarrow} Ab \stackrel{K\left(\blank,0\right)}{\longlongrightarrow} \cS,$$ where the functor $K\left(\blank,0\right)$ sends an abelian group to its underlying set. It's easy to check by direct calculation that this functor classifies the left fibration $\theta_0.$ The local system $\tau$ has a factorization
$$X \to \Pi_1\left(X\right) \stackrel{\tau'}{\longrightarrow} K\left(\Aut\left(A\right),1\right),$$ and by Lemma \ref{lem:pbfib}, the composite
\begin{equation}\label{eq:longfunctor}
\Pi_1\left(X\right) \stackrel{\tau'}{\longrightarrow} K\left(\Aut\left(A\right),1\right) \stackrel{\chi_A}{\longlongrightarrow} Ab \stackrel{K\left(\blank,0\right)}{\longlongrightarrow} \cS
\end{equation}
classifies the left fibration $$\Pi_1\left(X\right) \times_{K\left(\Aut\left(A\right),1\right) } \Aut\left(A\right) \ltimes \underline{A} \to \Pi_1\left(X\right).$$ By Proposition \ref{prop:colim}, it follows that the colimit of the composite (\ref{eq:longfunctor}) is the fibered product $\Pi_1\left(X\right) \times_{K\left(\Aut\left(A\right),1\right) } \Aut\left(A\right) \ltimes \underline{A},$ and hence one has an identification
$$\Pi_1\left(X\right) \times_{K\left(\Aut\left(A\right),1\right) } \Aut\left(A\right) \ltimes \underline{A} \simeq L_{X}\left(\tau,0\right),$$ using the notation from \cite[Definition 3.1]{coh}. By \cite[Corollary 4.6]{coh} there is a bijection between homotopy classes of lifts
$$\xymatrix@C=2.5cm{& L_X\left(\tau,0\right) \ar[d]\\
X \ar[r] \ar@{-->}[ru]& \Pi_1\left(X\right)}$$
and degree $0$ cohomology classes of $X$ with coefficients in $\tau.$ However, the space of such lifts is naturally homotopy equivalent to the space of lifts
$$\xymatrix@C=2.5cm{&  \Aut\left(A\right) \ltimes \underline{A} \ar^-{\theta_0}[d]\\
X \ar[r]^{\tau} \ar@{-->}[ru]& K\left(\Aut\left(A\right),1\right).}$$
\end{proof}

\begin{lemma}\label{lem:univpb}
Let $n>0$ be an integer. The following is a pullback diagram
$$\xymatrix@R=2cm@C=2cm{B\Aut\left(K\left(A,n\right)\right) \ar[d] \ar[r] & K\left(\Aut\left(A\right),1\right) \ar[d]_-{\theta_{n+1}} \\
K\left(\Aut\left(A\right),1\right) \ar[r]^-{\theta_{n+1}} & B\Aut\left(K\left(A,n+1\right)\right),}$$
where $B\Aut\left(K\left(A,n\right)\right) \to K\left(\Aut\left(A\right),1\right) =\Pi_1\left(B\Aut\left(K\left(A,n\right)\right)\right)$ is the canonical map from $B\Aut\left(K\left(A,n\right)\right)$ to its $1$-truncation. When $n=0,$ the following diagram is a pullback square
$$\xymatrix@R=2cm@C=2cm{B\left(\Aut\left(A\right)\ltimes A\right) \ar[d] \ar[r] & K\left(\Aut\left(A\right),1\right) \ar[d]_-{\theta_{1}} \\
K\left(\Aut\left(A\right),1\right) \ar[r]^-{\theta_{1}} & B\Aut\left(K\left(A,1\right)\right),}$$ where $\Aut\left(A\right)\ltimes A$ is the semi-direct product of groups.
\end{lemma}

\begin{proof}
Recall that the composite $$K\left(\Aut\left(A\right),1\right) \stackrel{\theta_{n+1}}{\longlongrightarrow} B\Aut\left(K\left(A,n+1\right)\right) \to \cS$$ is canonically equivalent to the composite
$$K\left(\Aut\left(A\right),1\right) \stackrel{\chi_A}{\longlongrightarrow} Ab \stackrel{K\left(\blank,n+1\right)}{\longlonglongrightarrow} \cS.$$ Since additionally, $$\theta_{n+1}:K\left(\Aut\left(A\right),1\right) \to B\Aut\left(K\left(A,n+1\right)\right)$$ is the left fibration associated with the canonical functor $$B\Aut\left(K\left(A,n+1\right)\right) \to \cS,$$ it follows from Lemma \ref{lem:pbfib} that we can identify the map $$K\left(\Aut\left(A\right),1\right)  \times_{B\Aut\left(K\left(A,n+1\right)\right)} K\left(\Aut\left(A\right),1\right) \to K\left(\Aut\left(A\right),1\right)$$ with the left fibration
$$\underset{K\left(\Aut\left(A\right),1\right)}\int \!\!\!\!\!\!\!\!\!\! K\left(\chi_A,n+1\right) \longrightarrow K\left(\Aut\left(A\right),1\right).$$ So, as an $\i$-category, we may identify the total space of the above fibration with the full subcategory of the slice category $\Pshi\left(K\left(\Aut\left(A\right),1\right)\right)/K\left(\chi_A,n+1\right)$ on those morphisms of the form $G \to K\left(\chi_A,n+1\right),$ with $G$ a representable presheaf. Since the groupoid $K\left(\Aut\left(A\right),1\right)$ has only one object $\star,$ and since $K\left(A,n+1\right)$ only has one object, it follows that $\underset{K\left(\Aut\left(A\right),1\right)}\int \!\!\!\!\!\!\!\!\!\! K\left(\chi_A,n+1\right)$ likewise has one object $W:y\left(\star\right) \to K\left(\chi_A,n+1\right)$ which corresponds under the Yoneda lemma to the unique object of $K\left(A,n+1\right).$ Now, by \cite[Proposition 5.5.5.12]{htt}, we have a pullback diagram
$$\xymatrix{\Hom\left(W,W\right) \ar[r] \ar[d] & \Hom\left(y\left(\star\right),y\left(\star\right)\right)\simeq \Aut\left(A\right) \ar[d]\\
\ast \ar[r] & \Hom\left(y\left(\star\right),K\left(\chi_A,n+1\right)\right)\simeq K\left(A,n+1\right).}$$
Since colimits are universal in $\cS$, it follows that
\begin{eqnarray*}
\Hom\left(W,W\right) &\simeq & \underset{\Aut\left(A\right)} \coprod \left(* \times_{K\left(A,n+1\right)} *\right)\\
&\simeq &  \underset{\Aut\left(A\right)} \coprod K\left(A,n\right)\\
&\simeq & \Aut\left(A\right) \times K\left(A,n\right)\\
&\simeq & \Aut\left(K\left(A,n\right)\right).
\end{eqnarray*}
It follows that 
$$\underset{K\left(\Aut\left(A\right),1\right)}\int \!\!\!\!\!\!\!\!\!\! K\left(\chi_A,n+1\right) \simeq B\left(\Aut\left(K\left(A,n\right)\right)\right).$$
Unwinding the definitions, we see that on mapping spaces, the left fibration has the effect
$$\Hom\left(W,W\right)\simeq \Aut\left(A\right) \times K\left(A,n\right) \to \Aut\left(A\right)$$
of first projection, and hence the left fibration can be identified with the canonical map 
$$B\Aut\left(K\left(A,n\right)\right) \to K\left(\Aut\left(A\right),1\right) =\Pi_1\left(B\Aut\left(K\left(A,n\right)\right)\right).$$ 

Finally, when $n=0,$ most of the above proof goes through, except of course that we do not have an identification of $\Aut\left(A\right) \times A$ with $\Aut\left(A\right),$ but rather the group structure on $\Aut\left(A\right) \times A$ is that of the semi-direct product.
\end{proof}

Suppose that $n>0.$ Given a map $f:X \to B\Aut\left(K\left(A,n\right)\right),$ we get an induced local system with coefficients in $A$ by considering the composite $$X \stackrel{f}{\longrightarrow} B\Aut\left(K\left(A,n\right)\right) \to \Pi_1\left(B\Aut\left(K\left(A,n\right)\right)\right)=K\left(\Aut\left(A\right),1\right).$$ Denoting by $\tau\left(f\right)$ the induced local system, $f$ itself can be identified with a section
$$\xymatrix@C=2.5cm{& B\Aut\left(K\left(A,n\right)\right)\ar[d]\\
X \ar@{-->}[ru]^-{f} \ar[r]_-{\tau\left(f\right)}& K\left(\Aut\left(A\right),1\right).}$$
However, by  Lemma \ref{lem:univpb}, $f$ can be identified with a section
$$\xymatrix@C=2.5cm{& K\left(\Aut\left(A\right),1\right) \ar[d]^-{\theta_{n+1}}\\
X \ar@{-->}[ru] \ar[r]_-{\theta_{n+1}\circ \tau\left(f\right)}& B\Aut\left(K\left(A,n+1\right)\right).}$$
In light of this, the following two corollary follows immediately from Proposition \ref{prop:twistedcoh}:

\begin{corollary}\label{cor:BAut rep 1}
Let $X$ be a space and let $n>0$ be an integer. Then there is a natural bijection between the set of homotopy classes of maps $$\left[X,B\Aut\left(K\left(A,n\right)\right)\right]$$ and the set of pairs $\left(\tau,\alpha\right),$ with $$\tau\in \left[X,K\left(\Aut\left(A\right),1\right)\right]$$ a local system on $X$ and $$\alpha \in H^{n+1}\left(X,\tau\right),$$ an $\left(n+1\right)^{st}$-cohomology class of $X$ with values in $\tau.$ Moreover, there is a natural bijection between the set of homotopy classes of maps $$\left[X,B\left(\Aut\left(A\right) \ltimes A\right)\right]$$ and the set of pairs $\left(\tau,\alpha\right),$ with $$\tau\in \left[X,K\left(\Aut\left(A\right),1\right)\right]$$ a local system on $X$ and $$\alpha \in H^{1}\left(X,\tau\right),$$ a degree $1$ cohomology class of $X$ with values in $\tau.$
\end{corollary}

\subsection{The $\i$-topos case}
Before turning our attention on how to define cohomology with coefficients in a local system of abelian groups on an $\i$-topos, we make a quick digression into the topic of local connectedness of $\i$-topoi. We will need the theory later to prove the comparison theorem we desire, and it is also relevant in distinguishing between two possibly different notions of a local system.

\subsubsection{Locally connected $\i$-topoi.}

\begin{definition}
An object $E$ in a topos $\cE$ is \textbf{connected} if whenever there is an isomorphism $E \cong U \coprod V$ in $\cE,$ then exactly one of $U$ and $V$ is not an initial object.
\end{definition}

\begin{remark}\label{rmk:connected}
An object $E$ in a topos $\cE$ is connected if and only if the functor $$\Hom_{\cE}\left(E,\blank\right):\cE \to \Set$$ preserves coproducts. (See Proposition \ref{prop: coproducts}.)
\end{remark}

\begin{definition}
A topos $\cE$ is \textbf{locally connected} if and only if every object $E$ in $\cE$ can be written as a coproduct of connected objects. (The initial object is an empty coproduct).
\end{definition}

\begin{lemma}\cite[Lemma C.3.3.6]{Johnstone}
A topos $\cE$ is \textbf{locally connected} if and only if the inverse image functor $$\Delta:\Set \to \cE$$ has a left adjoint $\Pi_0.$
\end{lemma}

\begin{remark}
Let $U$ be a connected object of a locally connected topos $\cE,$ and let $S$ be a set. Then we have:
\begin{eqnarray*}
\Hom\left(\Pi_0\left(U\right),S\right)&\cong& \Hom\left(U,\Delta\left(S\right)\right)\\
&\cong& \Hom\left(U,\underset{s \in S}\coprod 1\right)\\
&\cong& \underset{s \in S} \coprod \Hom\left(U,1\right)\\
&\cong& \underset{s \in S} \coprod *\\
&\cong& S,
\end{eqnarray*}
where the second to last isomorphism follows from Remark \ref{rmk:connected}, and thus $\Pi_0\left(U\right)\cong *.$ It follows that if $E=\underset{i \in I} \coprod U_i$ is a decomposition of $E$ into connected objects, then $$\Pi_0\left(E\right) \cong I,$$ hence the ``set of connected components of $E$'' is well defined up to isomorphism, and isomorphic to $\Pi_0\left(E\right).$
\end{remark}

\begin{definition}
A locally connected topos $\cE$ is \textbf{connected} if and only if the terminal object $1$ is connected. Equivalently, if and only if $$\Pi_0\left(\cE\right):=\Pi_0\left(1\right)\cong *.$$
\end{definition}

The following proposition is standard:

\begin{proposition}
Let $\sC$ be a locally connected Grothendieck site as in \cite[Section 9]{ArtinMazur}, then the topos of sheaves of sets $\Sh\left(\sC\right)$ is locally connected.
\end{proposition}

\begin{example}
Let $X$ be a locally connected topological space (in the strong sense that the each point $x$ has a neighborhood basis of connected open subsets). Then the open cover Grothendieck topology on the poset of open subsets $Op\left(X\right)$ is a locally connected site, and hence $\Sh\left(X\right)$ is locally connected. Any sheaf of sets $F$ on $X$ is the sections of a local homeomorphism $L\left(F\right) \to X,$ and such an $F$ is connected if and only if the space $L\left(F\right)$ is.
\end{example}

\begin{example}
Let $X$ be a locally Noetherian scheme. Then its small \'etale site $X_{\et}$ is locally connected (see \cite[I 6.1.9]{EGA1}). It follows that the small \'etale topos $\Sh\left(X_{\et}\right)$ is locally connected. Concretely, a representable sheaf $Y \to X$ in $\Sh\left(X_{\et}\right),$ i.e. an \'etale map from a scheme $Y,$ is connected if and only $Y$ is a connected scheme. More generally, as any \'etale sheaf over a scheme is representable by an \'etale map $P \to X$ from an algebraic space (with no separation conditions), a sheaf $F$ in $\Sh\left(X_{\et}\right)$ corresponding to such a map is connected if and only if the algebraic space $P$ is.
\end{example}

\begin{definition}
An $\i$-topos $\cE$ is \textbf{locally connected} if its underlying topos $\Disc\left(\cE\right)$ of discrete objects is a locally connected topos, where $\Disc\left(\cE\right)$ is the full subcategory of $\cE$ spanned by the $0$-truncated objects.
\end{definition}

\begin{remark}
It might be tempting to think an $\i$-topos is locally connected if and only if the inverse image functor $$\Delta:\iGpd \to \cE$$ has a left adjoint $\Pi_\i.$ However, this is a strictly stronger condition; an $\i$-topos satisfying this property is said to be \textbf{locally $\i$-connected}. For example, a locally connected space $X$ may not have $\Shi\left(X\right)$ locally $\i$-connected, but this will hold if $X$ is locally  contractible.
\end{remark}

\begin{definition}
An object $E$ in an $\i$-topos $\cE$ is \textbf{connected} if whenever there is an equivalence $E \simeq U \coprod V$ in $\cE,$ then exactly one of $U$ and $V$ is not an initial object.
\end{definition}

\begin{lemma}\label{lem:truncated connected}
An object $E$ in an $\i$-topos $\cE$ is connected if and only if its $0$-truncation $\pi_0\left(E\right)$ is connected in $\Disc\left(\cE\right).$
\end{lemma}

\begin{proof}
Suppose that $\pi_0\left(E\right)$ is connected, and we have $E \simeq U \coprod V.$ Then since $\pi_0$ is a left adjoint, we have $$\pi_0\left(E\right) \cong \pi_0\left(U\right) \coprod \pi_0\left(V\right).$$ So, without loss of generality, $\pi_0\left(U\right)$ is an initial object, and hence so is $U.$ Hence $E$ is connected.

Conversely, suppose that $E$ is connected and that $$\pi_0\left(E\right)\cong U \coprod V.$$ Then since colimits are universal, $$E \simeq E\times_{\pi_0\left(E\right)} U \coprod E\times_{\pi_0\left(E\right)} V,$$ and hence, without loss of generality, $E\times_{\pi_0\left(E\right)} U$ is an initial object. However, since $$E \to \pi_0\left(E\right)$$ is an epimorphism, it follows that so is $\emptyset=E\times_{\pi_0\left(E\right)} U \to U,$ therefore $U$ is initial.
\end{proof}

\begin{lemma}
Let $\cE$ be a locally connected $\i$-topos. Then any object $E$ can be written as a coproduct of connected objects.
\end{lemma}

\begin{proof}
Let $E$ be an object of a locally connected $\i$-topos. Then by definition, $\pi_0\left(E\right)$ is an object of a locally connected topos, hence we can write $$\pi_0\left(E\right)=\underset{i \in I} \coprod U_i$$ where each $U_i$ is connected. But then, since colimits are universal, it follows that $$E \simeq \underset{i \in I} \coprod E \times_{\pi_0\left(E\right)} U_i.$$ Now, since, $$\pi_0\left(E \times_{\pi_0\left(E\right)} U_i\right) \cong U_i,$$ each E $\times_{\pi_0\left(E\right)} U_i$ is connected by Lemma \ref{lem:truncated connected}.
\end{proof}

\begin{proposition}\label{prop: coproducts}
Let $E$ be an object of a locally connected $\i$-topos $\cE.$ Then $E$ is connected if and only if the functor $$\Hom_{\cE}\left(E,\blank\right):\cE \to \iGpd$$ preserves coproducts.
\end{proposition}

\begin{proof}
Let $E$ be connected and let $X=\underset{i \in I} \coprod X_i$ be an object of $\cE.$ Let $$f:E \to X$$ be a map in $\cE.$ Then since colimits are universal, we have $$E \simeq \underset{i \in I} \coprod E\times_X X_i.$$ Fix $j \in I$ and write $$E\simeq E\times_X X_j + \underset{i \ne j} \coprod E\times_X X_i.$$ Since $E$ is connected, 
only one the above factors can be non-initial. Moreover, we cannot have that $E\simeq E\times_X X_j$ is initial for all $j \in I,$ for this would imply that $E$ was initial. Now suppose by way of contradiction that there is $j \ne k$ in $I$ such that $E\simeq E\times_X X_j$ and $E\simeq E\times_X X_k$ are both non-initial. Then since $E$ is connected, $$\underset{i \ne j} \coprod E\times_X X_i$$ is initial, but $$\underset{i \ne j} \coprod E\times_X X_i=E\times_X X_k +
\underset{i \ne j,\\ i \ne k} \coprod E\times_X X_i,$$ and $E\times_X X_k$ is non-initial, which then leads to a contradiction. So $E\times_X X_i$ is non-initial for exactly one $i,$ and hence for this $i,$ $$E\times_X X_i \simeq E.$$ It follows that $f$ factors through the inclusion $X_i \to X.$ Hence $\Hom_{\cE}\left(E,\blank\right)$ preserves coproducts.

Conversely, suppose that $\Hom_{\cE}\left(E,\blank\right)$ preserves coproducts and that $E\simeq U \coprod V.$ Then $$\Hom_{\cE}\left(E,U \coprod V \right) \simeq \Hom_{\cE}\left(E,U \right) \coprod \Hom_{\cE}\left(E,V\right),$$ so the equivalence $$E \stackrel{\sim}{\longlongrightarrow} U \coprod V,$$ must factor through one of the factors, and hence the other factor must be initial.
\end{proof}

\subsubsection{Cohomology with coefficients in local systems in an $\i$-topos}

In this subsection we define local systems on an arbitrary $\i$-topos with coefficients in an abelian group and their associated cohomology groups. This is closely connected with the definition of twisted cohomology in an $\i$-topos; see e.g. \cite[Section 4]{principal}.

\begin{definition}
Let $\cE$ be an $\i$-topos. Let $A$ be an abelian group. Consider the groupoid $K\left(\Aut\left(A\right),1\right),$ and its associated stack in $\cE,$ $\Delta\left(K\left(\Aut\left(A\right),1\right)\right).$ A \textbf{local system with coefficients in $A$} on $\cE$ is a map $$\tau:1 \to \Delta\left(K\left(\Aut\left(A\right),1\right)\right)$$ in $\cE,$ where $1$ is the terminal object.
\end{definition}

Given a local system as above, there is an associated sheaf of abelian groups $\cF_\tau$ classified by $\tau.$ (By a sheaf of abelian groups, we mean an abelian group object in $\Disc\left(\cE\right).$) The main idea is that it is constructed by pulling back a canonical sheaf of abelian groups on $K\left(\Aut\left(A\right),1\right).$ We now explain in detail.

An abelian sheaf on the space $K\left(\Aut\left(A\right),1\right)$ is by definition a  sheaf of abelian groups on the $\i$-topos $\Pshi\left(K\left(\Aut\left(A\right),1\right)\right) \simeq \cS/\left(\Aut\left(A\right),1\right).$ Since $K\left(\Aut\left(A\right),1\right)$ is a groupoid, this is the same as specifying a functor $$K\left(\Aut\left(A\right),1\right) \to Ab,$$ to the category of abelian groups. We have already discussed such a functor $\chi_A$, namely the canonical functor sending $\star$ to $A$ and each automorphism of $A$ to itself. Lets denote this abelian sheaf by $\cF_A.$

We will now show that $\tau$ corresponds canonically to a geometric morphism $$\overline{\tau}:\cE \to \cS/\left(\Aut\left(A\right),1\right),$$ and then we will define $\cF_\tau$ as the pullback sheaf $\overline{\tau}^*\cF_A.$

Indeed, by \cite[Remark 6.3.5.10]{htt}, for any $\i$-topos $\cE,$ there is an equivalence of $\i$-categories
$$\cE \to \mathfrak{Top}_\i^{\et}/\cE$$ between $\cE$ and the $\i$-category of \'etale geometric morphisms over $\cE,$ which sends an object $E \in \cE$ to the canonical \'etale morphism $\cE/E \to \cE.$ Hence $\tau$ corresponds to a section of the \'etale geometric morphism $$\cE/\Delta\left(K\left(\Aut\left(A\right),1\right)\right) \to \cE$$ corresponding to the object $\Delta\left(K\left(\Aut\left(A\right),1\right)\right)$ of $\cE.$ By \cite[Proposition 6.3.5.8]{htt}, there is a pullback diagram in the $\i$-category of $\i$-topoi
$$\xymatrix{\cE/\Delta\left(K\left(\Aut\left(A\right),1\right)\right) \ar[r] \ar[d] & \cS/K\left(\Aut\left(A\right),1\right)\ \ar[d] \\
\cE \ar[r] & \cS,}$$
so the aforementioned section can be identified with a lift 
$$\xymatrix{& \cS/K\left(\Aut\left(A\right),1\right)\ \ar[d] \\
\cE \ar[r] \ar@{-->}[ru]^{\overline{\tau}} & \cS,}$$
and since $\cS$ is the terminal $\i$-topos, we conclude that the data of the local system $\tau$ and the geometric morphism $$\overline{\tau}:\cE \to \cS/K\left(\Aut\left(A\right),1\right)$$ are equivalent, or more precisely:

\begin{proposition}
The construction just explained produces an equivalence of $\i$-groupoids
$$\Hom_{\cE}\left(1,K\left(\Aut\left(A\right),1\right)\right) \simeq \Hom_{\mathfrak{Top}_\i}\left(\cE,\cS/K\left(\Aut\left(A\right),1\right)\right)$$ between local systems with coefficients in $A$ on $\cE$ and geometric morphisms from $\cE$ into $\cS/K\left(\Aut\left(A\right),1\right).$
\end{proposition}

The following corollary follows immediately from \cite[Remark 7.1.6.15]{htt}:

\begin{corollary}
There is an equivalence of $\i$-groupoids
$$\Hom_{\cE}\left(1,K\left(\Aut\left(A\right),1\right)\right) \simeq \Hom_{\Pro\left(\cS\right)}\left(\Shape\left(\cE\right),K\left(\Aut\left(A\right),1\right)\right).$$
\end{corollary}

\begin{example}
If $X$ is a scheme, then a local system on its small \'etale $\i$-topos $\Shi\left(X_{\et}\right)$ is the same a morphism $$\tau:\Pi^{\et}_\i\left(X\right) \to K\left(\Aut\left(A\right),1\right)$$ from its \'etale fundamental $\i$-groupoid to $K\left(\Aut\left(A\right),1\right).$
\end{example}

\begin{definition}
Let $\tau:1 \to K\left(\Aut\left(A\right),1\right)$ be a local system with coefficients in $A$ on $\cE.$ Then the abelian sheaf $\cF_\tau:=\overline{\tau}^*\cF_A$ is the abelian sheaf \textbf{classified by the local system $\tau$.}
\end{definition}

\begin{remark}\label{rmk:pullback set}
By the proof of Proposition \ref{prop:0th coh}, the object in $\cS/K\left(\Aut\left(A\right),1\right)$ corresponding to the underlying sheaf of sets of $\cF_A$ can be identified with the functor of groupoids $$\theta_0:\Aut\left(A\right) \ltimes \underline{A} \to K\left(\Aut\left(A\right),1\right).$$ Moreover, by construction, there is a factorization of $\overline{\tau}$ of the form
$$\cE \stackrel{\cE/\tau}{\longlongrightarrow} \cE/\Delta\left(K\left(\Aut\left(A\right),1\right)\right) \to \cS/K\left(\Aut\left(A\right),1\right).$$ Unwinding the definitions, one sees that the underlying sheaf of sets of $\cF_\tau,$ $\underline{\cF_\tau}$ fits in a pullback diagram
$$\xymatrix{\underline{\cF_\tau} \ar[r] \ar[d] & \Delta\left(\Aut\left(A\right) \ltimes \underline{A}\right) \ar[d]^-{\Delta\left(\theta_0\right)}\\
1 \ar[r]^-{\tau} & \Delta\left(K\left(\Aut\left(A\right),1\right)\right).}$$
\end{remark}

\begin{definition}
Let $A$ be an abelian group and $\cE$ an $\i$-topos. A \textbf{locally constant sheaf with values in $A$} on $\cE$ is an abelian sheaf $\cF$ on $\cE$ such that there are objects $\left(U_i\right)_{i \in I}$ in $\cE$ such that the canonical map $$\underset{i \in I} \coprod U_i \to 1$$ is an epimorphism, and such that the pullback of $\cF$ to each slice topos $\cE/U_i$ is isomorphic to the constant abelian sheaf with value $A.$
\end{definition}

\begin{remark}\label{rmk:ordinary locally constant}
Let $\sC$ be a small category equipped with a Grothendieck topology. Then any object $E$ in $\Shi\left(\sC\right)$ admits an epimorphism from a coproduct of representables, hence a locally constant abelian sheaf in $\Shi\left(\sC\right)$ can be identified with a classical locally constant abelian sheaf on $\sC.$
\end{remark}

\begin{proposition}
Let $\tau:1 \to K\left(\Aut\left(A\right),1\right)$ be a local system in an $\i$-topos $\cE.$ Then the abelian sheaf $\cF_\tau$ classified by $\tau$ is a locally constant sheaf with values in $A.$
\end{proposition}

\begin{proof}
Notice that $\Delta\left(K\left(\Aut\left(A\right),1\right)\right)$ is the classifying stack for $\Aut\left(A\right)$-torsors. In particular, the universal $\Aut\left(A\right)$-torsor $$1=\Delta\left(*\right)\stackrel{\Delta\left(\star\right)}{\longlongrightarrow} \Delta\left(K\left(\Aut\left(A\right),1\right)\right)$$ is an epimorphism. Consider the following pullback diagram
$$\xymatrix{P_\tau \ar[r] \ar[d] & 1 \ar[d]^-{\Delta\left(\star\right)}\\
1 \ar[r]^-{\tau} & \Delta\left(K\left(\Aut\left(A\right),1\right)\right).}$$
The map $P_\tau \to 1$ is an epimorphism and by Remark \ref{rmk:pullback set}, we can identify the underlying sheaf of sets of the pullback of $\cF_\tau$ to $\cE/P_\tau$ 
as the map $Q \to P_{\tau}$ in the following pullback diagram
$$\xymatrix{Q \ar[rr] \ar[d] &  &\Delta\left(\Aut\left(A\right) \ltimes \underline{A}\right) \ar[d]^-{\theta_0}\\
P_{\tau} \ar[r] & 1 \ar[r]^-{\tau} & \Delta\left(K\left(\Aut\left(A\right),1\right)\right).}$$
Note that the pullback diagram defining $P_\tau$ in particular commutes, so the above pullback diagram may also be computed as
$$\xymatrix@C=2cm{Q \ar[rr] \ar[d] &  &\Delta\left(\Aut\left(A\right) \ltimes \underline{A}\right) \ar[d]^-{\theta_0}\\
P_{\tau} \ar[r] & 1 \ar[r]^-{\Delta\left(\star\right)} & \Delta\left(K\left(\Aut\left(A\right),1\right)\right).}$$
By the proof of Proposition \ref{prop:0th coh}, and the fact that $\Delta$ preserves finite limits, the following diagram is also a pullback
$$\xymatrix@C=2cm{\Delta\left(\underline{A}\right) \ar[d] \ar[r] & \Delta\left(\Aut\left(A\right) \ltimes \underline{A}\right) \ar[d]^-{\theta_0}\\
\Delta\left(\ast\right) \ar[r]^-{\Delta\left(\star\right)} & \Delta\left(K\left(\Aut\left(A\right),1\right)\right).}$$
It follows that $Q\simeq P_\tau \times \Delta\left(\underline{A}\right) \to P_{\tau},$ i.e. the pullback of $\underline{\cF_{\tau}}$ to $\cE/P_\tau$ is equivalent to the constant sheaf $\Delta^{P_\tau}\left(\underline{A}\right).$ Hence the same is true for the abelian sheaf, i.e. the pullback of $\cF_{\tau}$ to $\cE/P_\tau$ is constant with value $A,$ and hence $\cF_\tau$ is locally constant.
\end{proof}

\begin{remark}
For a general $\i$-topos $\cE$, it is not true that every locally constant sheaf with values in $A$ is classified by a local system $$\tau:1 \to K\left(\Aut\left(A\right),1\right),$$ however it is true if $\cE$ is locally connected. The reason is as follows. Let $\cA$ be any abelian sheaf. We say for an object $E$ of $\cE,$ that an abelian sheaf $\cF$ on $\cE/E$ is \emph{locally isomorphic to $\cA$} (or a twisted form of $\cA$) if there is an epimorphism $$\underset{i} \coprod E_i \to E,$$ such that the restriction of $\cF$ to each $E_i$ is isomorphic to the restriction of $\cA$ to each $E_i.$ It is a classical fact that the groupoid of abelian sheaves on $E$ locally isomorphic to $\cA$ is equivalent to the groupoid of morphism $\Hom\left(E,B\left(\underline{\Aut}\left(\cA\right)\right)\right),$ where $\underline{\Aut}\left(\cA\right)$ is the automorphism sheaf of $\cA,$ c.f. \cite[Chapter III, Section 4]{Milne}. Now consider the constant sheaf $\Delta\left(A\right)$ for $A$ an abelian group. To show that locally constant sheaves with values in $A$ are classified by morphisms into $$\Delta\left(K\left(\Aut\left(A\right),1\right)\right)=\Delta\left(B\Aut\left(A\right)\right),$$
it suffices to show that $$\Delta\left(B\Aut\left(A\right)\right) \simeq B \underline{\Aut}\left(\Delta\left(A\right)\right).$$
Notice that $$B\Aut\left(A\right) \simeq \underset{n \in \Delta^{op}} \colim \Aut\left(A\right)^n,$$ and since $\Delta$ preserves colimits and finite limits we have
$$\Delta\left(B\Aut\left(A\right)\right) \simeq \underset{n \in \Delta^{op}} \colim \Delta\left(\Aut\left(A\right)\right)^n \simeq B\left(\Delta\left(\Aut\left(A\right)\right)\right).$$ So it suffices to show that $$\Delta\left(\Aut\left(A\right)\right) \cong \underline{\Aut}\left(\Delta\left(A\right)\right),$$ when $\cE$ is locally connected. This follows readily from the following observation: Let $S$ be any set, and denote by $\Delta_{\Disc}$ the inverse image functor of the essentially unique geometric morphism of $1$-topoi $$\Disc\left(\cE\right) \to \Set.$$ Then since $\Delta_{\Disc}$ has a left adjoint $\Pi_0,$ it preserves limits and we have
\begin{eqnarray*}
\Delta\left(\Hom\left(S,S\right)\right) &\cong& \Delta\left(\underset{S}\prod S\right)\\
&\cong& \underset{S} \prod \Delta\left(S\right)\\
&\cong& \underline{\mathbf{End}}\left(\Delta\left(S\right)\right).
\end{eqnarray*}
\end{remark}

\begin{definition}
Let $\cE$ be an $\i$-topos, $A$ an abelian group, and $$\tau:1 \to \Delta\left(K\left(\Aut\left(A\right),1\right)\right)$$ a local system on $\cE.$ The \textbf{$n^{th}$  cohomology group of $\cE$ with values in $\tau$} is $$\pi_0\Hom_{\cE}\left(1,K\left(\cF_{\tau},n\right)\right),$$ where $K\left(\cF_{\tau},n\right)$ is the $n^{th}$ Eilenberg-MacLane object of the abelian sheaf $\cF_{\tau}$ classified by $\tau.$
\end{definition}

\begin{remark}
Let $\sC$ be a small category equipped with a Grothendieck topology. Let $\tau$ be a local system on $\Shi\left(\sC\right)$ with values in an abelian group $A.$ By Remark \ref{rmk:ordinary locally constant}, we can identify the abelian sheaf classified by $\tau$ with a classical locally constant sheaf of abelian groups $\cF_{\tau}$ on $\sC.$ Furthermore, by \cite[Remark 7.2.2.17]{htt}, we can identify the $n^{th}$ cohomology group of $\Shi\left(\sC\right)$ with values in $\tau$ as just defined with the $n^{th}$ cohomology group of $\cF_{\tau}$ as computed using classical sheaf cohomology.
\end{remark}

\begin{theorem}\label{thm: classifying space infinity topos}
Let $\cE$ be an $\i$-topos, $A$ and abelian group and $$\tau:1 \to \Delta\left(K\left(\Aut\right),1\right)$$ a local system on $\cE$ with values in $A.$ The $0^{th}$ cohomology group of $\cE$ with coefficients in $\tau$ is isomorphic to 
$$\pi_0\left(\Hom_{\cE/\Delta\left(K\left(\left(\Aut\right),1\right)\right)}\left(\tau,\Delta\left(\theta_0\right)\right)\right),$$
i.e. $\pi_0$ of the space of lifts
$$\xymatrix{ & \Delta\left(\Aut\left(A\right) \ltimes \underline{A}\right) \ar[d]^-{\Delta\left(\theta_0\right)}\\
1 \ar[r]_-{\tau} \ar@{-->}[ru] & \Delta\left(K\left(\Aut\left(A\right),1\right)\right)}$$
equipped with the group structure induced from that of $A.$ Moreover, for $n>0,$ the $n^{th}$ cohomology group of $\cE$ with coefficients in $\tau$ can be identified with
$$\pi_0\left(\Hom_{\cE/\Delta\left(B\Aut\left(K\left(A,n\right)\right)\right)}\left(\Delta\left(\theta_n\right) \circ \tau,\Delta\left(\theta_n\right)\right)\right),$$
i.e. $\pi_0$ of the space of lifts
$$\xymatrix@R=3cm@C=2cm{ & & \Delta\left(K\left(\Aut\left(A\right),1\right)\right) \ar[d]^-{\Delta\left(\theta_n\right)}\\
1 \ar[r]_-{\tau} \ar@{-->}[rru] & \Delta\left(K\left(\Aut\left(A\right),1\right)\right) \ar[r]_-{\Delta\left(\theta_n\right)} & \Delta\left(B\Aut\left(K\left(A,n\right)\right)\right).}$$
\end{theorem}

\begin{proof}
The statement about the $0^{th}$ cohomology group follows immediately from Remark \ref{rmk:pullback set}.

Now suppose that $n>0.$ Recall that $\cF_A$ is the abelian sheaf on $\cS/K\left(\Aut\left(A\right),1\right)$ corresponding to the functor $$\chi_A:K\left(\Aut\left(A\right),1\right) \to Ab,$$ and $\cF_{\tau}$ is by definition $\overline{\tau}^*\cF_A,$ where $\overline{\tau}:\cE \to \cS/K\left(\Aut\left(A\right),1\right)$ is the geometric morphism induced by $\tau.$ Denote by $K\left(\cF_A,n\right)$ the $n^{th}$ Eilenberg-MacLane object of $\cF_A$ in $\cS/K\left(\Aut\left(A\right),1\right).$ By \cite[Remark 6.5.1.4]{htt}, it follows that $$\overline{\tau}^*K\left(\cF_A,n\right) \simeq K\left(\cF_{\tau},n\right).$$

Under the equivalence $$\cS/K\left(\Aut\left(A\right),1\right)\simeq \Fun\left(K\left(\Aut\left(A\right),1\right),\cS\right),$$ $K\left(\cF_A,n\right)$ corresponds to the composite $$K\left(\Aut\left(A\right),1\right) \to Ab \stackrel{K\left(\blank,n\right)}{\longlongrightarrow} \cS,$$ which means that $K\left(\cF_A,n\right)$ in $\cS/K\left(\Aut\left(A\right),1\right)$ is the left fibration classified by the above composite functor. Recall this functor also factors as the composite
$$K\left(\Aut\left(A\right),1\right) \stackrel{\theta_n}{\longrightarrow} B\Aut\left(K\left(A,n\right)\right) \stackrel{\Theta_n}{\longrightarrow} \cS,$$ where $\Theta_n:B\Aut\left(K\left(A,n\right)\right) \to \cS$ is the natural functor which in fact classifies the universal $K\left(A,n\right)$-fibration $$\theta_n:K\left(\Aut\left(A\right),1\right) \to B\Aut\left(K\left(A,n\right)\right).$$ Denote by $$\cS/\theta_n:\cS/K\left(\Aut\left(A\right),1\right) \to \cS/B\Aut\left(K\left(A,n\right)\right)$$ the geometric morphism induced by $\theta_n,$ then regarding $\theta_n$ as an object of $\cS/B\Aut\left(K\left(A,n\right)\right),$ we have a canonical identification
$$\left(\cS/\theta_{n}\right)^*\left(\theta_n\right)\simeq K\left(\cF_A,n\right).$$ And hence $K\left(\cF_{\tau},n\right)$ can be identified with the pullback of $\theta_n$ along the geometric morphism $$\cE \stackrel{\overline{\tau}}{\longrightarrow} \cS/K\left(\Aut\left(A\right),1\right) \stackrel{\cS/\theta_n}{\longlonglongrightarrow} \cS/B\Aut\left(K\left(A,n\right)\right).$$ Unwinding the definitions, this means that we have a pullback diagram in $\cE$
$$\xymatrix@C=2.5cm@R=2cm{K\left(\cF_\tau,n\right) \ar[r] \ar[d] & \Delta\left(K\left(\Aut\left(A\right),1\right)\right) \ar[d]^-{\Delta\left(\theta_n\right)} \\
1 \ar[r]^-{\Delta\left(\theta_n\right) \circ \tau} & \Delta\left(B\Aut\left(K\left(A,n\right)\right)\right).}$$ The result now follows.
\end{proof}

The following corollary is proved in the same was as Corollary \ref{cor:BAut rep 1}:

\begin{corollary}\label{cor:BAut rep 2}
Let $\cE$ be an $\i$-topos, $A$ an abelian group, and $n>0$ be an integer. Then there is a natural bijection between the set of global sections
$$\pi_0\Gamma\left(\Delta\left(B\Aut\left(K\left(A,n\right)\right)\right)\right)$$
and the set of pairs $\left(\tau,\alpha\right),$ with $$\tau \in \pi_0\Gamma\left(K\left(\Aut\left(A\right),1\right)\right)$$ a local system on $\cE$ and $$\alpha \in H^{n+1}\left(\cE,\tau\right),$$ an $\left(n+1\right)^{st}$-cohomology class of $\cE$ with values in $\tau.$ Moreover, there is a natural bijection between the set global sections $$\pi_0\Gamma\left(\Delta\left(B\left(\Aut\left(A\right) \ltimes A\right)\right)\right)$$ and the set of pairs $\left(\tau,\alpha\right),$ with $$\tau\in \pi_0\Gamma\left(\Delta\left(K\left(\Aut\left(A\right),1\right)\right)\right)$$ a local system on $\cE$ and $$\alpha \in H^{1}\left(\cE,\tau\right),$$ a degree $1$ cohomology class of $\cE$ with values in $\tau.$
\end{corollary}

\section{A profinite comparison theorem}
In this section, we  extend the results of \cite{ArtinMazur} to show that the profinite \'etale homotopy type of any higher stack on the site of affine schemes of finite type over $\mathbb{C}$ agrees with the profinite homotopy type of its underlying topological stack.

We start by recalling some notions and results from Section 3 of \cite{knhom}.

Let $\Top$ be the category of topological spaces and let $\TopCs$ denote the full subcategory on all those spaces which are contractible and locally contractible spaces which are homeomorphic to a subspace of $\mathbb{R}^n$ for some $n.$ Denote by $\TopC$ the following subcategory of topological spaces:

\begin{definition}
A topological space $T$ is in $\TopC$ if $T$ has an open cover $\left(U_\alpha \hookrightarrow T\right)_\alpha$ such that each $U_\alpha$ is an object of $\TopCs.$
\end{definition}

The reason for decorating the notation with ``$\mathbb{C}$'' is that $\TopC$ is a good recipient for the analytification functor from complex schemes. Recall from \cite{toen-vaquie} that there is an analytification functor
$$\left(\blank\right)_{an}:\Sch^{LFT}_\bC \to \Top,$$
from schemes locally of finite type over $\bC$ to topological spaces, and this functor preserves finite limits. When $X$ is a scheme, $X_{an}=X\left(\bC\right)$ is its space of $\bC$-points equipped with the complex analytic topology. $X_{an}$ is locally (over any affine) a triangulated space by \cite{Lo}, so in particular $X_{an}$ is locally contractible, and since $X_{an}$ is locally cut-out of $\mathbb{C}^n$ by polynomials, so it follows that $X_{an}$ is in $\TopC.$

In \cite[Section 20]{No1}, Noohi extends the analytification functor to a left exact functor
$$\left(\blank\right)_{top}:\mathbf{A}\!\mathfrak{lgSt}^{LFT}_{\mathbb{C}} \to \mathfrak{TopSt}$$ from Artin stacks locally of finite type over $\mathbb{C}$ to topological stacks. For $\cX$ an Artin stack, $\cX_{top}$ is called its \emph{underlying topological stack}. In \cite[Theorem 3.1 and Corollary 3.11]{knhom}, we extend this further to a left exact colimit preserving functor
$$\left(\blank\right)_{top}:\Shi\left(\Aff,\mbox{\'et}\right) \to \Hshi\left(\TopC\right)$$
from $\i$-sheaves on the \'etale site of affine schemes of finite type over $\bC,$ to hypersheaves on $\TopC$ (with respect to the open cover topology). For $\cX$ any $\i$-stack on $\left(\Aff,\mbox{\'et}\right),$ we refer to $\cX_{top}$ as its \emph{underlying stack on $\TopC$}.

\begin{remark}\label{rmk: topos small}
Even though $\TopC$ is not a small category, $\Hshi\left(\TopC\right)$ is an $\i$-topos, since there is a canonical equivalence of $\i$-categories
$$\Hshi\left(\TopC\right) \simeq \Hshi\left(\TopCs\right).$$ See \cite[Section 3.1]{knhom}.
\end{remark}

In Section 3.2 of \cite{knhom}, we also prove the following theorem:

\begin{theorem}\cite[Proposition 3.12 and Corollary 3.13]{knhom}
There is a colimit preserving functor $$\Pi_\i:\Hshi\left(\TopC\right) \to \cS$$ which sends space $X$ in $\TopC$, to its underlying weak homotopy type.
\end{theorem}

The functor $\Pi_\i$ is called the \textbf{fundamental $\i$-groupoid} functor. (This is an extension of the results of \cite{No2}.)

We now state our main result:\\

For any $\i$-sheaf $F$ on $\left(\Aff,\mbox{\'et}\right),$ there is an equivalence of profinite spaces $$\widehat{\Pi}^{\et}_\i\left(F\right) \simeq \widehat{\Pi}_\i\left(F_{top}\right),$$ between the profinite \'etale homotopy type of $F$ and the profinite completion of the homotopy type of the underlying stack $F_{top}$ on $\TopC$ (see Theorem \ref{thm: main}).\\

We will need a few preliminaries:\\

Notice that there is a canonical functor $$\TopC \to \mathfrak{Top}_\i.$$ This factors as the canonical inclusion $$\TopC \hookrightarrow \Top$$ followed by the canonical functor
\begin{eqnarray*}
\Sh:\Top &\to& \mathfrak{Top}\\
T &\mapsto& \Sh\left(T\right),
\end{eqnarray*}
from topological spaces to topoi (which is fully faithful when restricted to sober spaces), followed by the canonical inclusion $$\mathfrak{Top} \hookrightarrow \mathfrak{Top}_\i$$ identifying topoi with $1$-localic $\i$-topoi. Since the poset of open subsets of a topological space has finite limits, it follows from \cite[Proposition 6.4.5.4]{htt} that the total composite sends a topological space $T$ to the $\i$-topos $\Shi\left(T\right)$ of $\i$-sheaves over $T.$ Denote by $$\mathbb{H}\mathrm{yp}:\mathfrak{Top}_\i \to \mathfrak{Top}_\i$$ the hypercompletion functor (see Remark \ref{rmk:hypercompletion functor}).

Recall that by Remark \ref{rmk: topos small}, there is a canonical equivalence of $\i$-categories $$\Hshi\left(\TopC\right) \simeq \Hshi\left(\TopCs\right).$$

The following lemma's proof is completely analogous to that of Lemma \ref{lem:smalletale}. We leave the details to the reader:

\begin{lemma}
There exists a colimit preserving functor
$$\Hshi\left(\blank\right):\Hshi\left(\TopC\right) \to \mathfrak{Top}_\i$$ which sends a representable sheaf $y\left(T\right)$ for $T$ a topological space to the $\i$-topos of hypersheaves on $T.$
\end{lemma}

\begin{lemma}\label{lem: hypershape homotopy}
The following diagram commutes up to equivalence:
$$\xymatrix@C=2.5cm{\Hshi\left(\TopC\right) \ar[r]^-{\Hshi\left(\blank\right)} \ar[d]_-{\Pi_\i} & \mathfrak{Top}_\i \ar[d]^-{\Shape}\\
\cS \ar@{^{(}->}[r]^-{j} & \Pro\left(\cS\right).}$$
\end{lemma}

\begin{proof}
Recall that the there is a canonical equivalence $$\Hshi\left(\TopC\right)\simeq \Hshi\left(\Top^sC\right).$$ Since all the functors in the above diagram preserve colimits, it suffices by \cite[Proposition 5.5.4.20, Theorem 5.1.5.6]{htt} to prove that there is a natural equivalence of functors
$$\Shape \circ \Hshi\left(\blank\right) \circ y \simeq j \circ \Pi_\i \circ y$$ where $$y:\Top^sC \hookrightarrow \Hshi\left(\Top^sC\right)$$ is the Yoneda embedding. Let $T$ be an object of $\Top^sC.$ In particular, $T$ is locally contractible. By the proof of Proposition \ref{prop: locally contractible shape}, we have a canonical identification $$\Shape\left(\Hshi\left(T\right)\right)\simeq j\left(\Pi_\i\left(T\right)\right),$$ and by construction, there is a canonical equivalence $$\Shape\left(\Hshi\left(T\right)\right) \simeq \Shape \circ \Hshi\left(\blank\right)\left(y\left(T\right)\right).$$
\end{proof}

\begin{proposition}\label{prop: main 1}
There is a canonical natural transformation

$$ \xygraph{!{0;(6.5,0):(0,0.15)::}
{\Shi\left(\Aff,\mbox{\'et}\right)}="a" [r] {\mathfrak{Top}_\i}="b"
"a":@/^{1.5pc}/"b"^-{\Hshi\left(\blank\right)\circ \left(\blank\right)_{top}}|(.4){}="l"
"a":@/_{1.5pc}/"b"_-{\Shi\left(\left(\blank\right)_{\et}\right)}
"l" [d(.3)]  [r(0.1)] :@{=>}^{\xi} [d(.5)]} $$

Such that the induced natural transformation
$$\xymatrix@C=2.5cm{\Shape^{\Prof} \circ \Hshi\left(\blank\right)\circ \left(\blank\right)_{top}  \ar@{=>}[r]^-{\Shape^{\Prof}\left(\xi\right)}& \Shape^{\Prof} \circ \Shi\left(\left(\blank\right)_{\et}\right)=\widehat{\Pi}^{\et}_\i}$$ is an equivalence.
\end{proposition}

To prove the above proposition, since all the functors involved are colimit preserving, by \cite[Proposition 5.5.4.20, Theorem 5.1.5.6]{htt} it suffices to prove the result after restricting the functors to affine schemes of finite type over $\bC.$ The affine assumption will not play a role, so we will establish the result for any scheme $X$ of finite type over $\bC.$

Following \cite[expos\'e XI.4]{SGA4}:

Denote by $\Top^{\et}/X_{an}$ the category of local homeomorphisms over $X_{an}.$ Let $$\alpha:X_{\et} \to \Top^{\et}/X_{an}$$ be the restriction of the analytification functor; it sends an \'etale map of schemes $f:Y \to X$ to the local homeomorphism $f_{an}:Y_{an} \to X_{an}.$ Note also that via the \'etal\'e space construction, there is a canonical equivalence of categories $\Top^{\et}/X_{an} \simeq \Sh\left(X_{an}\right).$ Since $\Sh\left(X_{an}\right)$ has enough points, and since $\alpha$ is left exact, it follows that $\alpha$ is flat, and hence by \cite[B3.2.7]{Johnstone}, $\alpha$ induces a geometric morphism $$\varphi:\Sh\left(X_{an}\right) \to \Psh\left(X_{\et}\right).$$ Explicitly, $\varphi^*$ is the left Kan extension $\Lan_y\left(\alpha\right)$ of $\alpha$ along the Yoneda embedding, and for $F$ a sheaf on $X_{an}$ corresponding to a local homeomorphism $LF \to X_{an},$ $\varphi_*\left(F\right)$ evaluated on an \'etale morphism $Y \to X$ is $\Hom_{X_{an}}\left(Y_{an},LF\right).$ The functor $\alpha$ sends \'etale covering families in $X_{\et}$ to families of jointly surjective local homeomorphisms. These are exactly the effective epimorphisms in $\Top^{\et}/X_{an}.$ Hence, identifying $\alpha$ with a functor $X_{\et} \to \Sh\left(X_{an}\right),$ we see that $\alpha$ pulls back sheaves on $\Sh\left(X_{an}\right)$ equipped with the canonical topology to sheaves, since the canonical topology $\Sh\left(X_{an}\right)$ is precisely generated by jointly epimorphic families. It follows that $\varphi_*\left(F\right)$ is always a sheaf, and hence $\varphi^*$ restricts to a left exact colimit preserving functor $$\varepsilon_X^*:\Sh\left(X_{\et}\right) \to \Sh\left(X_{an}\right),$$ hence constitutes a geometric morphism $$\varepsilon_X:\Sh\left(X_{an}\right) \to \Sh\left(X_{\et}\right).$$ Since both $X_{\et}$ and the poset of open subsets of $X_{an}$ have finite limits, this canonically extends to a geometric morphism of $\i$-topoi $$\varepsilon_X:\Shi\left(X_{an}\right) \to \Shi\left(X_{\et}\right)$$ by \cite[Proposition 6.4.5.4]{htt}.

We define $\xi_X$ as the composition
$$\Hshi\left(X_{an}\right) \stackrel{\epsilon_X}{\longlongrightarrow} \Shi\left(X_{an}\right) \stackrel{\varepsilon_X}{\longlongrightarrow} \Shi\left(X_{\et}\right).$$

Suppose that $f:X \to Y$ is a morphism of schemes, and consider the following diagram of categories:
$$\xymatrix{\Sh\left(X_{\et}\right) \ar[r]^-{\varepsilon^*_X} & \Sh\left(X_{an}\right)\\
\Sh\left(Y_{\et}\right) \ar[u]^-{f^*} \ar[r]^-{\varepsilon^*_Y} & \Sh\left(Y_{an}\right) \ar[u]^-{f^*_{an}}.}$$
Since all the functors involved preserve colimits, and since analytification preserves finite limits, it follows that there is a canonical $2$-morphism $$\varepsilon\left(f\right):\varepsilon^*_X \circ f^* \stackrel{\sim}{\Rightarrow} f^*_{an} \circ \varepsilon^*_Y.$$ That is to say, $\varepsilon\left(f\right)$ represents a $2$-morphism in the $\left(2,1\right)$-category $\mathfrak{Top}$ of topoi, making the following diagram commute
$$\xymatrix{\Sh\left(X_{an}\right) \ar[d] \ar[r]^-{\varepsilon_X} & \Sh\left(X_{\et}\right) \ar[d]\\
\Sh\left(Y_{an}\right)  \ar[r]^-{\varepsilon_Y} & \Sh\left(Y_{\et}\right).}$$
Moreover, it is easy to check that the various geometric morphisms $\varepsilon_X$ together with these $2$-morphisms assemble into a lax natural-transformation
$$ \xygraph{!{0;(6.5,0):(0,0.15)::}
{\Aff}="a" [r] {\mathfrak{Top}.}="b"
"a":@/^{1.5pc}/"b"^-{\Sh\left(\blank\right)\circ \left(\blank\right)_{an}}|(.4){}="l"
"a":@/_{1.5pc}/"b"_-{\Shi\left(\left(\blank\right)_{\et}\right).}
"l" [d(.3)]  [r(0.1)] :@{=>}^{\varepsilon} [d(.5)]} $$
(The necessary coherency conditions follow by a similar argument by pasting diagrams.)
By abuse of notation, composition with the canonical inclusion $$\mathfrak{Top} \hookrightarrow \mathfrak{Top}_\i$$ induces a natural transformation
$$ \xygraph{!{0;(6.5,0):(0,0.15)::}
{\Aff}="a" [r] {\mathfrak{Top}}="b"
"a":@/^{1.5pc}/"b"^-{\Sh\left(\blank\right)\circ \left(\blank\right)_{an}}|(.4){}="l"
"a":@/_{1.5pc}/"b"_-{\Shi\left(\left(\blank\right)_{\et}\right).}
"l" [d(.3)]  [r(0.1)] :@{=>}^{\varepsilon} [d(.5)]} $$
Finally, by composing with the counit $$\epsilon: q \circ \mathbb{H}yp \Rightarrow id_{\mathfrak{Top}_\i}$$ of the coreflective subcategory of hypercomplete $\i$-topoi (Remark \ref{rmk:hypercompletion functor}), we get a natural transformation
$$ \xygraph{!{0;(6.5,0):(0,0.15)::}
{\Aff}="a" [r] {\mathfrak{Top}_\i.}="b"
"a":@/^{1.5pc}/"b"^-{\Hshi\left(\blank\right)\circ \left(\blank\right)_{an}}|(.4){}="l"
"a":@/_{1.5pc}/"b"_-{\Shi\left(\left(\blank\right)_{\et}\right)}
"l" [d(.3)]  [r(0.1)] :@{=>}^{\xi} [d(.5)]} $$
Again, since all the functors involved in the satement of Proposition \ref{prop: main 1} are colimit preserving, by \cite[Proposition 5.5.4.20, Theorem 5.1.5.6]{htt} this natural transformation lifts to one of the form
$$ \xygraph{!{0;(6.5,0):(0,0.15)::}
{\Shi\left(\Aff,\mbox{\'et}\right)}="a" [r] {\mathfrak{Top}_\i,}="b"
"a":@/^{1.5pc}/"b"^-{\Hshi\left(\blank\right)\circ \left(\blank\right)_{top}}|(.4){}="l"
"a":@/_{1.5pc}/"b"_-{\Shi\left(\left(\blank\right)_{\et}\right)}
"l" [d(.3)]  [r(0.1)] :@{=>}^{\xi} [d(.5)]} $$
and to prove that it is an equivalence after applying $\Shape^{\Prof}$, it suffices to show that each geometric morphism $\xi_X$ is a profinite homotopy equivalence, when $X$ is a scheme of finite type over $\bC$. Note that by Proposition \ref{prop:hypersame},
$$\epsilon_X:\Hshi\left(X_{an}\right) \to \Shi\left(X_{an}\right)$$ is a profinite homotopy equivalence for all $X$, so it suffices to prove that each geometric morphism $$\varepsilon_X:\Shi\left(X_{an}\right) \to \Shi\left(X_{\et}\right)$$ is a profinite homotopy equivalence as well.

Let us fix a scheme $X$ of finite type over $\bC$ and denote $\varepsilon_X$ from now on by $\varepsilon.$

The main ingredient in showing that $\varepsilon$ is a profinite homotopy equivalence is the following classical result from \cite{SGA4}:

\begin{theorem}\cite[expos\'e XI.4 Theorem 4.3, Theorem 4.4, and expos\'e XVI.4, Theorem 4.1]{SGA4}\label{thm:SGA}
Let $X$ be a scheme of finite type over $\bC.$ Then
\begin{itemize}
\item[1)] The analytification functor $\alpha:X_{\et} \to \Top^{\et}/X_{an}$ induces an equivalence of categories between finite \'etale maps over $X$  and finite covering spaces of $X_{an}.$
\item[2)] $\varepsilon$ induces an isomorphism in cohomology with coefficients in any local system of finite abelian groups.
\end{itemize}
\end{theorem}

Recall that from Remark \ref{rmk:concprof} that $\varepsilon$ is a profinite homotopy equivalence if and only if for every $\pi$-finite space $V,$ the induced map 
$$\Gamma_{\et}\Delta^{\et}\left(V\right) \to \Gamma_{an}\Delta^{an}\left(V\right)$$ is an equivalence, where
$$\xymatrix@C=2.5cm{\Shi\left(X_{\et}\right) \ar@<-0.5ex>[r]_-{\Gamma_{\et}} & \cS \ar@<-0.5ex>[l]_-{\Delta^{\et}}}$$
and
$$\xymatrix@C=2.5cm{\Shi\left(X_{an}\right) \ar@<-0.5ex>[r]_-{\Gamma_{an}} & \cS \ar@<-0.5ex>[l]_-{\Delta^{an}}}$$
are the unique geometric morphisms to the terminal $\i$-topos $\cS.$
Our method of proof will be to first establish this for connected $\pi$-finite spaces by using induction on Postnikov towers.

First we will need a few lemmas:

\begin{lemma}\label{lem: finite cover}
Let $X$ be a connected scheme of finite type over $\bC,$ and consider the geometric morphism $$\varepsilon:\Sh\left(X_{an}\right) \to \Sh\left(X_{\et}\right).$$ Suppose that $f:Y \to X$ is a finite \'etale map. Then $$\varepsilon_*\left(f_{an}\right)\cong f.$$
\end{lemma}

\begin{proof}
Let $g:Z \to X$ be any \'etale map from a scheme, i.e. an object of $X_{\et}.$ Then we have that $$\varepsilon_*\left(f_{an}\right)\left(g\right)=\Hom_{\Top^{\et}/X_{an}}\left(g_{an},f_{an}\right).$$ Consider the pullback diagram
$$\xymatrix{Z_{an} \times_{X_{an}} Y_{an} \ar[r] \ar[d]_-{g_{an}^*\left(f_{an}\right)} & Y_{an} \ar[d]^-{f_{an}}\\
Z_{an} \ar[r]^{g_{an}} & X_{an}.}$$
Since analytification preserves finite limits, we have $$Z_{an} \times_{X_{an}} Y_{an}\cong \left(Z\times_X Y\right)_{an}$$ and $g_{an}^*\left(f_{an}\right)=g^*\left(f\right)_{an}.$ Since we have a pullback diagram, 
$$\Hom_{\Top^{\et}/X_{an}}\left(g_{an},f_{an}\right)\cong \Hom_{\Top^{\et}/Z_{an}}\left(id_{Z_{an}},g^*\left(f\right)_{an}\right),$$ and $g^*\left(f\right)_{an}$ is a finite covering projection so, by Theorem \ref{thm:SGA}, $1),$ 
$$\Hom_{\Top^{\et}/Z_{an}}\left(id_{Z_{an}},g^*\left(f\right)_{an}\right)\cong \Hom_{Z_{\et}}\left(id_Z,g^*\left(f\right)\right).$$ Finally, we can identify $\Hom_{Z_{\et}}\left(id_Z,g^*\left(f\right)\right)$ with $\Hom_{X_{\et}}\left(g,f\right),$ by using the pullback diagram for $Z \times_X Y.$ The result now follows.
\end{proof}

\begin{lemma}\label{lem: torsors}
Let $G$ be a finite group, and $X$ a scheme of finite type over $\bC.$ Then the analytification functor $$X_{\et} \to \Top^{\et}/X_{an}$$ induces an equivalence of categories between the category of $G$-torsors on $X$ and the category of principal $G$-bundles over $X_{an}.$
\end{lemma}

\begin{proof}
Consider the geometric morphism $$\varepsilon:\Sh\left(X_{an}\right) \to \Sh\left(X_{\et}\right).$$ Suppose that $\cP$ is a $G$-torsor on $X,$ with underlying map $P \to X,$ which is a finite \'etale map. The analytification functor sends this to $P_{an} \to X_{an}$ which is a finite covering map, and since the analytification functor preserves finite limits, $P_{an}$ inherits a fiber-wise $G$-action in such a way that it makes it a principal $G$-bundle. It follows from Theorem \ref{thm:SGA}, $1),$ that this construction produces a fully faithful functor from $G$-torsors on $X$ to principal $G$-bundles over $X_{an}.$ It suffices to show this is essentially surjective. Suppose that $\cP'$ is  is a principal $G$-bundle with underlying finite covering space $\pi':P' \to X_{an}.$ By Theorem \ref{thm:SGA}, $1),$ we may assume this finite covering space is of the form $\pi_{an}$ for $\pi:Q \to X$ a finite \'etale map, and by Lemma \ref{lem: finite cover}, we may further assume that  $\pi=\varepsilon_*\left(\pi'\right).$ Denote by $\Delta_{an}\left(G\right)$ the constant sheaf of groups, and regard $\cP'$ as a $\Delta_{an}\left(G\right)$-torsor, i.e. as a sheaf $E_P$ in $\Sh\left(X_{an}\right)$ equipped with a $\Delta_{an}\left(G\right)$-action such that the canonical map $$\Delta_{an}\left(\underline{G}\right) \times E_P \to E_P \times E_P$$ is an isomorphism, where  $\underline{G}$ denotes the underlying set of $G.$ Note that $\Delta^{\et}\left(\underline{G}\right)$ is representable by the finite \'etale map $$pr_2:\underline{G} \times X \to X,$$ and the analytification of this map is $$pr_2:\underline{G} \times X_{an} \to X_{an},$$ which is the \'etal\'e space of $\Delta_{an}\left(\underline{G}\right).$ It follows from Lemma \ref{lem: finite cover} and the fact that $\varepsilon_*$ preserves limits that $$\varepsilon_*\left(\Delta\left(G\right)\right)\cong\Delta_{\et}\left(G\right).$$ Moreover, $\varepsilon_*$ carries $\Delta_{an}\left(G\right)$-torsors in $\Sh\left(X_{an}\right)$ to $\Delta^{\et}\left(G\right)$-torsors in $\Sh\left(X_{\et}\right).$ The result now follows.
\end{proof}

\begin{lemma}\label{lem: K(A,n)}
Let $f:\cE \to \cF$ be a geometric morphism of $\i$-topoi and let $\cA$ be an abelian sheaf in $\cF.$ Suppose that for all $n \ge 0,$ $f$ induces an isomorphism in sheaf cohomology groups
$$H^n\left(\cF,\cA\right) \stackrel{\sim}{\longrightarrow} H^n\left(\cE,f^*\cA\right).$$
Then for all $n,$ then induced map $$\Gamma_{\cF}\left(K\left(\cA,n\right)\right) \to \Gamma_{\cE}\left(K\left(f^*\cA,n\right)\right)$$ is an equivalence of $\i$-groupoids.
\end{lemma}

\begin{proof}
First of all, it follows immediately from \cite[Remark 6.5.1.4]{htt} that $$f^*K\left(\cA,n\right)=K\left(f^*\cA,n\right),$$ which explains the induced map above; it is induced by the functor $f^*$ (since $f^*$ preserves the terminal object). Notice that for any abelian sheaf $\cB$ on an $\i$-topos $\cX,$ for $n>0,$ $K\left(\cB,n\right)$ has the structure of a grouplike $\mathbb{E}_\i$-object in $\cX,$ and consequently, $$\Gamma_{\cX}\left(K\left(\cB,n\right)\right)=\Hom_{\cX}\left(1,K\left(\cB,n\right)\right)$$ is a grouplike $\mathbb{E}_\i$-space. Let $e:1 \to K\left(\cB,n\right)$ be the group unit. We can identify $e$ with a map from the one-point space
$$* \stackrel{e}{\longrightarrow} \Hom_{\cX}\left(1,K\left(\cB,n\right)\right).$$ Let $$*\stackrel{g}{\longrightarrow} \Hom_{\cX}\left(1,K\left(\cB,n\right)\right)$$ be any other base point of the space $\Hom_{\cX}\left(1,K\left(\cB,n\right)\right).$ Denote by $\left[g\right]$ the image of $g$ in $\pi_0\left( \Hom_{\cX}\left(1,K\left(\cB,n\right)\right)\right).$ Note that $\pi_0\left( \Hom_{\cX}\left(1,K\left(\cB,n\right)\right)\right)$ is a group, since the $\mathbb{E}_\i$-space $\Hom_{\cX}\left(1,K\left(\cB,n\right)\right)$ is grouplike. Now, for any $n >0,$ multiplication with $\left[g\right]$ induces an isomorphism $$\pi_n\left(\Hom_{\cX}\left(1,K\left(\cB,n\right)\right),e\right) \stackrel{\sim}{\longrightarrow} \pi_n\left(\Hom_{\cX}\left(1,K\left(\cB,n\right)\right),g\right).$$ From this discussion, it suffice to prove in our situation, that $$\Gamma_{\cF}\left(K\left(\cA,n\right)\right) \to \Gamma_{\cE}\left(K\left(f^*\cA,n\right)\right)$$ induces an isomorphism on $\pi_0$ and on all higher homotopy groups \emph{at the canonical base point $e$.} So it suffices to show that for all $k>0$ the induced map
$$\pi_0\Omega^k_{e}\left(\Gamma_{\cF}\left(K\left(\cA,n\right)\right)\right) \to \pi_0\Omega^k_{e}\left(\Gamma_{\cE}\left(K\left(f^*\cA,n\right)\right)\right)$$ is an isomorphism.

Notice that the canonical base point $e$ is in the image of the global sections functor $\Gamma_{\cF},$ i.e.
$$e=\Gamma_{\cF}\left(e\right):\Gamma_{\cF}\left(1\right) \to \Gamma_{\cF}\left(K\left(\cA,n\right)\right),$$ and since $\Gamma_{\cF}$ preserves finite limits, it then follows that $$\Omega^k_{e}\left(\Gamma_{\cF}\left(K\left(\cA,n\right)\right)\right)\simeq \Gamma_{\cF}\left(\Omega^k_{e}\left(\left(K\left(\cA,n\right)\right)\right)\right).$$ When $k \le n,$ $\Omega^k_{e}\left(\left(K\left(\cA,n\right)\right)\right)\simeq K\left(\cA,n-k\right)$ and for $k>n,$ it's the terminal object. Consequently, we have that
$$\pi_k\left(\Gamma_{\cF}\left(K\left(\cA,n\right)\right),e\right) \cong \pi_0\Omega^k_{e}\left(\Gamma_{\cF}\left(K\left(\cA,n\right)\right)\right) \cong H^{n-k}\left(\cF,\cA\right)$$ for $k \le n,$ and otherwise is zero, and similarly for $\pi_k\left(\Gamma_{\cE}\left(K\left(f^*\cA,n\right)\right),e\right).$ The result now follows.
\end{proof}

\begin{proposition}\label{prop: main theorem connected}
Let $X$ be a connected scheme of finite type over $\bC.$ Then $$\varepsilon:\Shi\left(X_{an}\right) \to \Shi\left(X_{\et}\right)$$ is a profinite homotopy equivalence.
\end{proposition}

\begin{proof}
It suffices to prove that for every $\pi$-finite space, the induced map $$\Gamma_{\et}\Delta^{\et}\left(V\right) \to \Gamma_{an}\Delta^{an}\left(V\right)$$ is an equivalence of $\i$-groupoids. Denote by $\sC$ the full subcategory of $\cS$ spanned by all spaces $V$ for which the above map is an equivalence. Note that since the functors $\Gamma_{\et},\Delta^{\et},\Gamma_{an},\Delta^{an}$ all preserve finite limits, if follows that $\sC$ is closed under finite limits in $\cS.$ Note by Theorem \ref{thm:SGA}, $2),$ together with Lemma \ref{lem: K(A,n)} it follows in particular that $\sC$ contains all Eilenberg-MacLane spaces of the form $K\left(A,n\right),$ with $A$ a finite abelian group. Also, it follows from Lemma \ref{lem: torsors} that $\sC$ contains all $K\left(G,1\right)$ for all finite groups $G.$

Fix a finite abelian group $A$ and let $n >0$ be an integer. We claim that $B\Aut\left(K\left(A,n\right)\right)$ is also in $\sC.$ Let us establish this claim. We have already seen that the canonical map $$\Gamma_{\et}\Delta^{\et}\left(K\left(\Aut\left(A\right),1\right)\right) \to \Gamma_{an}\Delta^{an}\left(K\left(\Aut\left(A\right),1\right)\right)$$ is an equivalence of $\i$-groupoids. Consider the canonical map $$\psi:B\Aut\left(K\left(A,n\right)\right) \to K\left(\Aut\left(A\right),1\right)$$ induced by identifying $K\left(\Aut\left(A\right),1\right)$ as the $1$-truncation of $B\Aut\left(K\left(A,n\right)\right).$ It suffices to prove that for every base point $$\tau:* \to \Gamma_{\et}\Delta^{\et}\left(K\left(\Aut\left(A\right),1\right)\right),$$ the induced maps between the (homotopy) fiber of $$\Gamma_{\et}\Delta^{\et}\left(\psi\right)$$ over $\tau$ and the (homotopy) fiber of $$\Gamma_{an}\Delta^{an}\left(\psi\right)$$ over $\varepsilon^*\tau$ is an equivalence of $\i$-groupoids. Let $F_n,$ denote the fiber of $$\Gamma_{\et}\Delta^{\et}\left(\psi\right)$$ over $\tau,$ i.e. the pullback
$$\xymatrix{F_n \ar[r] \ar[d] & \Hom_{\Shi\left(X_{\et}\right)}\left(1,\Delta^{\et}\left(B\Aut\left(K\left(A,n\right)\right)\right)\right) \ar[d]^-{\Gamma_{\et}\Delta^{\et}\left(\psi\right)}\\
\ast \ar[r]^-{\tau} & \Hom_{\Shi\left(X_{\et}\right)}\left(1,\Delta^{\et}\left(K\left(\Aut\left(A\right),1\right)\right)\right).}$$
By \cite[Proposition 5.5.5.12]{htt}, we have a canonical identification $$F_n\simeq \Hom_{\Shi\left(X_{\et}\right)/\Delta^{\et}\left(K\left(\Aut\left(A\right),1\right)\right)}\left(\tau,\Delta^{\et}\left(\psi\right)\right).$$ 
The latter space of maps is the space of lifts
$$\xymatrix{1 \times_{\Delta^{\et}\left(K\left(\Aut\left(A\right),1\right)\right)} \Delta^{\et}\left(B\Aut\left(K\left(A,n\right)\right)\right)  \ar[d] \ar[r]& \Delta^{\et}\left(B\Aut\left(K\left(A,n\right)\right)\right) \ar[d]^-{\Delta^{\et}\left(\psi\right)}\\
1 \ar@{-->}[ru] \ar[r]^-{\tau} & \Delta^{\et}\left(K\left(\Aut\left(A\right),1\right)\right)}$$
which is canonically equivalent to the space $$\Gamma^{\et}\left(1 \times_{\Delta^{\et}\left(K\left(\Aut\left(A\right),1\right)\right)} \Delta^{\et}\left(B\Aut\left(K\left(A,n\right)\right)\right)\right).$$
Since $\Delta^{\et}$ preserves finite limits, it follows from Lemma \ref{lem:univpb} that the following is a pullback diagram in $\Shi\left(X_{\et}\right)$:
$$\xymatrix@R=2cm@C=2cm{\Delta^{\et}\left(B\Aut\left(K\left(A,n\right)\right)\right) \ar[d]_-{\Delta^{\et}\left(\psi\right)} \ar[r] & \Delta^{\et}\left(K\left(\Aut\left(A\right),1\right)\right) \ar[d]^-{\Delta^{\et}\left(\theta_{n+1}\right)} \\
\Delta^{\et}\left(K\left(\Aut\left(A\right),1\right)\right) \ar[r]^-{\Delta^{\et}\left(\theta_{n+1}\right)} & \Delta^{\et}\left(B\Aut\left(K\left(A,n+1\right)\right)\right).}$$
In light of this, by the pullback square at the end of the proof of Theorem \ref{thm: classifying space infinity topos}, we have a canonical identification $$1 \times_{\Delta^{\et}\left(K\left(\Aut\left(A\right),1\right)\right)} \Delta^{\et}\left(B\Aut\left(K\left(A,n\right)\right)\right) \simeq K\left(\cF_{\tau},n+1\right),$$ where $\cF_{\tau}$ is the abelian sheaf classified by the local system $\tau.$ In summary, we have that the fiber of $\Gamma_{\et}\Delta^{\et}\left(\psi\right)$ over $\tau$ can canonically be identified with $\Gamma_{\et}\left(K\left(\cF_{\tau},n+1\right)\right).$ Hence the induced map between fibers can be identified with the induced map $$\Gamma_{\et}\left(K\left(\cF_{\tau},n+1\right)\right) \to \Gamma_{an}\left(K\left(\cF_{\varepsilon^*\tau},n+1\right)\right).$$ By Theorem \ref{thm:SGA}, $2),$ for all $n,$ the induced map 
$$H^n\left(\Shi\left(X_{\et}\right),\cF_{\tau}\right) \stackrel{\sim}{\longrightarrow} H^n\left(\Shi\left(X_{an}\right),\cF_{\varepsilon^*\tau}\right)$$ is an isomorphism. The claim now follows from Lemma \ref{lem: K(A,n)}.

Since $X$ is connected, it follows that so is $X_{an},$ and hence the terminal objects of both $\Shi\left(X_{an}\right)$ and $\Shi\left(X_{\et}\right)$ are connected, by Lemma \ref{lem:truncated connected}. It follows then from Proposition \ref{prop: coproducts} that both $\Gamma_{\et}$ and $\Gamma_{an}$ preserve coproducts, and hence $\sC$ is closed under coproducts in $\cS.$ This reduces our job to checking that $$\Gamma_{\et}\Delta^{\et}\left(V\right) \to \Gamma_{an}\Delta^{an}\left(V\right)$$ is an equivalence for all \emph{connected} $\pi$-finite spaces. 

Lets prove by induction on $n$ that $\sC$ contains all connected $n$-truncated $\pi$-finite spaces. We have already established that this holds for $n=1.$ Suppose by hypothesis that $n \ge 2$ and $\sC$ contains all $\left(n-1\right)$-truncated connected $\pi$-finite spaces. We wish to show that $\sC$ contains all $n$-truncated connected $\pi$-finite spaces. Let $Z$ be such a space. Denote by $Z_{n-1}$ the $\left(n-1\right)^{st}$-truncation of $Z.$ Then $Z \to Z_{n-1}$ has fiber $K\left(\pi_n\left(Z\right),n\right).$ Let $A$ be the abelian group $\pi_n\left(Z\right).$ Then by Proposition \ref{prop: univ KAN fibration}, we have a pullback square
$$\xymatrix{Z \ar[d] \ar[r] & K\left(\Aut\left(A\right),1\right) \ar[d]^-{\theta_n}\\
Z_{n-1} \ar[r] & B\Aut\left(K\left(A,n\right)\right).}$$ Since $\sC$ is stable under finite limits, it now follows that $Z$ is in $\sC$ as well.
\end{proof}

\begin{proposition}
Let $X$ be a scheme of finite type over $\bC.$ Then $$\varepsilon:\Shi\left(X_{an}\right) \to \Shi\left(X_{\et}\right)$$ is a profinite homotopy equivalence.
\end{proposition}

\begin{proof}
Let $X=\underset{\alpha} \coprod X_{\alpha},$ with each $X_{\alpha}$ a connected scheme. For any space $V,$ 
\begin{eqnarray*}
\Gamma^{\et}\Delta^{\et}\left(V\right)&=& \Hom_{\Shi\left(X_{\et}\right)}\left(\underset{\alpha}\coprod X_{\alpha},\Delta^{\et}\left(V\right)\right)\\
&\simeq&\underset{\alpha} \prod \Hom_{\Shi\left(X_{\et}\right)}\left(X_{\alpha},\Delta^{\et}\left(V\right)\right)\\
&\simeq& \underset{\alpha} \prod \Hom_{\Shi\left(\left(X_\alpha\right)_{\et}\right)}\left(1,\Delta_\alpha^{\et}\left(V\right)\right)\\
&\simeq & \underset{\alpha} \prod \Gamma_\alpha^{\et}\Delta_\alpha^{\et}\left(V\right).
\end{eqnarray*}
The analytification of $X$ is $$X_{an}=\underset{\alpha} \coprod \left(X_{\alpha}\right)_{an}$$ and each $X_{\alpha}$ is connected as a topological space. By analogous reasoning as above we have $$\Gamma^{an}\Delta^{an}\left(V\right) \simeq \underset{\alpha} \prod \Gamma_\alpha^{an}\Delta_\alpha^{an}\left(V\right).$$ The result now follows from Proposition \ref{prop: main theorem connected}.
\end{proof}

This establishes the proof of Proposition \ref{prop: main 1}. We now prove our main theorem:

\begin{theorem}\label{thm: main}
The following diagram commutes up to equivalence:
$$\xymatrix@C=2.5cm@R=2cm{\Shi\left(\Aff,\mbox{\'et}\right) \ar[r]^-{\widehat{\Pi}^{\et}_\i} \ar[d]_-{\left(\blank\right)_{top}} & \Profs\\
\Hshi\left(\TopC\right) \ar[r]^-{\Pi_\i} & \cS. \ar[u]_-{\widehat{\left(\blank\right)}}}$$
In particular, for any $\i$-sheaf $F$ on $\left(\Aff,\mbox{\'et}\right),$ there is an equivalence of profinite spaces $$\widehat{\Pi}^{\et}_\i\left(F\right) \simeq \widehat{\Pi}_\i\left(F_{top}\right),$$ between the profinite \'etale homotopy type of $F$ and the profinite completion of the homotopy type of the underlying stack $F_{top}$ on $\TopC$.
\end{theorem}

\begin{proof}
By Proposition \ref{prop: main 1}, we have an equivalence
$$\xymatrix@C=2.5cm{\Shape^{\Prof} \circ \Hshi\left(\blank\right)\circ \left(\blank\right)_{top}  \ar@{=>}[r]^-{\Shape^{\Prof}\left(\xi\right)}_-{\mbox{\resizebox{!}{0.08in}{$\sim$}}}& \widehat{\Pi}^{\et}_\i.}$$ Note that by definition we have 
$$\Shape^{\Prof}= i^* \circ \Shape,$$ so  
$$Shape^{\Prof} \circ \Hshi\left(\blank\right)\circ \left(\blank\right)_{top}=i^* \circ Shape \circ \Hshi\left(\blank\right)\circ \left(\blank\right)_{top}.$$ By Lemma \ref{lem: hypershape homotopy} we have an equivalence $$Shape \circ \Hshi\left(\blank\right)\ \simeq j \circ \Pi_\i.$$ Furthermore, by definition the profinite completion functor $$\widehat{\left(\blank\right)}:\cS \to \Prof$$ is $i^* \circ j,$ so finally 
$$\Shape^{\Prof} \circ \Hshi\left(\blank\right)\circ \left(\blank\right)_{top}\simeq \widehat{\left(\blank\right)} \circ \Pi_\i \circ \left(\blank\right)_{top}.$$
\end{proof}

\begin{corollary} \label{cor: main}
Let $\cX$ be an Artin stack locally of finite type over $\bC,$ then there is an equivalence of profinite spaces $$\widehat{\Pi}^{\et}_\i\left(\cX\right) \simeq \widehat{\Pi}_\i\left(\cX_{top}\right),$$ between the profinite \'etale homotopy type of $\cX$ and the profinite completion of the homotopy type of the underlying topological stack $\cX_{top}$.
\end{corollary}

\begin{example}
Consider the moduli stack $\mathcal{M}_{g,n}$ of proper smooth curves of genus $g$ with $n$ marked points, and let $\Gamma_{g,n}$ be the mapping class group of a surface of genus $g$ with $n$ marked points. Fix an embedding $$\overline{\mathbb{Q}}\hookrightarrow \mathbb{C}.$$ Then it was shown in \cite{Oda} that the homotopy type of the analytification of $$\mathcal{M}_{g,n} \otimes \overline{\mathbb{Q}}$$ is that of $B\Gamma_{g,n}.$ It follows that there is an equivalence of profinite spaces $$\widehat{\Pi}^{\et}_\i\left(\mathcal{M}_{g,n} \otimes \overline{\mathbb{Q}}\right) \simeq \widehat{B\Gamma_{g,n}}.$$ An analogous result was shown in \cite{Oda} using the machinery of \'etale homotopy type of Friedlander, and the notion of profinite completion of Artin-Mazur. Similarly, it follows from \cite{compact} that $$\widehat{\Pi}^{\et}_\i\left(\overline{\mathcal{M}_{g,n}} \otimes \overline{\mathbb{Q}}\right) \simeq \widehat{B\mathcal{CL}_{g,n}},$$ where $\overline{\mathcal{M}_{g,n}}$ is the Deligne-Mumford compactification, and $\mathcal{CL}_{g,n}$ is the Charney-Lee category.
\end{example}

\begin{example}
Consider the moduli stack of elliptic curves $\mathcal{M}_{ell}$, and fix an embedding $$\overline{\mathbb{Q}}\hookrightarrow \mathbb{C}.$$ Then it is shown in \cite{polarised} than the homotopy type of the analytification is that of $BSL\left(2,\mathbb{Z}\right),$ from which it follows that $$\widehat{\Pi}^{\et}_\i\left(\mathcal{M}_{ell} \otimes \overline{\mathbb{Q}}\right) \simeq \widehat{BSL\left(2,\mathbb{Z}\right)}.$$ As above, an analogous result was shown in the same paper, but using the machinery of  \'etale homotopy type of Friedlander, and the notion of profinite completion of Artin-Mazur.
\end{example}

\begin{example}
Let $G$ be an algebraic group over $\mathbb{C}$. Then
 $$\widehat{\Pi}^{\et}_\i\left(BG\right) \simeq \widehat{BG_{an}}.$$
\end{example}

\begin{example}
Let $X$ be a fine saturated log scheme locally of finite type over $\mathbb{C},$ and let $\sqrt[\i]{X}$ be its infinite-root stack in the sense of \cite{TV}. (This is a pro-object in algebraic stacks). It was shown in \cite{knhom} that the homotopy type of the underlying (pro-)topological stack, after profinite completion agrees with the Kato-Nakayama space $X_{log}$ of $X$ in the sense of \cite{KN}. It thus follows that $$\widehat{\Pi}^{\et}_\i\left(\sqrt[\i]{X}\right) \simeq \widehat{X_{log}}.$$ As the Kato-Nakayama space is only defined for log schemes over $\mathbb{C},$ this suggests that the infinite root stack could be a suitable replacement for it in positive characteristics.
\end{example}

\bibliography{profinite}
\bibliographystyle{hplain}
\end{document}